\documentclass[10pt,oneside]{amsart}

\usepackage{inputenc}

\usepackage{amsthm,ulem}

\usepackage{stackengine}[2013-09-11]

 \usepackage{amsmath}
\usepackage{amsfonts}
\usepackage{amssymb}
\usepackage{amsthm}
\usepackage{colonequals}
\usepackage{mathrsfs}

\usepackage[usenames]{xcolor}

\usepackage[hyperfootnotes=false]{hyperref}

\title[]{On neighborhoods of embedded complex tori}
\author[]{Xianghong Gong$^{\dag}$}

\address{Department of Mathematics,
	University of Wisconsin-Madison, Madison, WI 53706, U.S.A.}
\email{gong@math.wisc.edu}

\author{Laurent Stolovitch$^{\dag\dag}$}
\address{CNRS and Laboratoire J.-A. Dieudonn\'e
	U.M.R. 7351, Universit\'e C\^ote d'Azur, Parc Valrose
	06108 Nice Cedex 02, France}
\email{Laurent.stolovitch@univ-cotedazur.fr}
\thanks{$^{\dag}$Partially supported by a grant from the Simons
	Foundation (award number: 505027) and NSF grant DMS-2054989. $^{\dag\dag}$Research of L. Stolovitch has been supported by the French government through the UCAJEDI Investments in the Future project managed by the National Research Agency (ANR) with the reference number ANR-15-IDEX-01.
}


\newcommand{\dist}{\operatorname{dist}}

\newtheorem{thm}{Theorem}[section]

\newtheorem{prop}[thm]{Proposition}
\newtheorem{lemma}[thm]{Lemma}
\newtheorem{remark}[thm]{Remark}

\newcommand{\crd}{{({\bf Corrected})}}

\newcommand{\diag}{\operatorname{diag}}

\newtheorem{defn}[thm]{Definition}
\newtheorem{exmp}[thm]{Example}

	\newcommand{\ga}{\begin{gather}}
	\newcommand{\ega}{\end{gather}}
	\newcommand{\gan}{\begin{gather*}}
	\newcommand{\egan}{\end{gather*}}
	\newcommand{\al}{\begin{align}}
	\newcommand{\eal}{\end{align}}
	\newcommand{\aln}{\begin{align*}}
	\newcommand{\ealn}{\end{align*}}
	\newcommand{\eq}[1]{\begin{equation}\label{#1}}
	\newcommand{\eeq}{\end{equation}}

	\newcommand{\ci}{~\cite}

	\newcommand{\f}[2]{\frac{#1}{#2}}
\newcommand{\C}{{\mathbb C}}
\newcommand{\R}{{\mathbb R}}
\newcommand{\Z}{{\mathbb Z}}

\newcommand{\cyl}{{\widetilde C}}
	\newcommand{\nn}{{\bf N}}

	\newcommand{\ov}{\overline}

	\newcommand{\IM}{\operatorname{Im}}

	\newcommand{\cL}{\mathcal}

	\newcommand{\all}{\alpha}

	\newcommand{\del}{\delta}
	\newcommand{\Del}{\Delta}
	\newcommand{\var}{\varphi}
	\newcommand{\e}{\epsilon}
	\newcommand{\om}{\omega}
	\newcommand{\Om}{\Omega}

	\newcommand{\la}{\lambda}

	\newcommand{\pd}{\partial}

	\newcommand{\re}[1]{(\ref{#1})}
	\newcommand{\rea}[1]{$(\ref{#1})$}
	\newcommand{\rl}[1]{Lemma~\ref{#1}}
	
	\newcommand{\rp}[1]{Proposition~\ref{#1}}
	\newcommand{\rt}[1]{Theorem~\ref{#1}}
	\newcommand{\rd}[1]{Definition~\ref{#1}}

	\newcommand{\rla}[1]{Lemma~$\ref{#1}$}

	\newcounter{pp}
	\newcommand{\bpp}{\begin{list}{$\hspace{-1em}(\alph{pp})$}{\usecounter{pp}}}
		\newcommand{\epp}{\end{list}}
	
	\newcounter{ppp}
	\newcommand{\bppp}{\begin{list}{$\hspace{-1em}(\roman{ppp})$}{\usecounter{ppp}}}
		\newcommand{\eppp}{\end{list}}
	
	\def\beq{\begin{equation}}
	\def\eeq{\end{equation}}

\begin{document}
\begin{abstract}
	The goal of the article is to show that an  $n$-dimensional complex torus embedded in  a complex manifold of dimensional $n+d$, with a split tangent bundle, has neighborhood biholomorphic a neighborhood of the zero section in its normal bundle, provided the latter has
(locally constant) Hermitian transition functions and satisfies a {\it non-resonant Diophantine} condition.
\end{abstract}		
\maketitle	
\hspace{6cm}{\it In memory of Jean-Pierre Demailly}
\vspace{.5cm}
\setcounter{thm}{0}\setcounter{equation}{0}
	\section{Introduction}
In this paper we
show the following
\begin{thm}\label{thm1.1}
Let $C$ be an $n$-dimensional complex torus embedded in a complex manifold $M$ of dimensional $n+d$.  Assume that $T_CM$, the restriction of $TM$ on $C$,   splits as $TC\oplus N_C$. Suppose that the normal bundle of $C$ in $M$ admits transition functions that are
 Hermitian matrices and satisfy a {\it non-resonant Diophantine} condition $($see Definition~$\ref{dioph})$. Then a neighborhood of  $C$ in $M$ is biholomorphic to a neighborhood of   the zero section in the normal bundle.
\end{thm}

We first describe the organization of the proof of our main theorem.

A complex torus $C$ can be identified with the quotient of $\C^n$ by a lattice $\Lambda$ spanned by the standard unit vectors $e_1,\dots, e_n$ in $\C^n$ and $n$ additional vectors $e'_1,\dots, e'_n$ in $\C^n$, where $\IM e'_1,\dots, \IM e'_n$ are linearly independent vectors in $\R^n$. Let $\Lambda'$ be the lattice in the cylinder $\cyl:=\R^n/\Z^n+i\R^n$ spanned by $  e_1',\dots,   e'_n\mod \Z^n$.   There are two coverings for the  torus $C=\C^n/\Lambda=\cyl/{\Lambda'}$: the universal covering $\pi\colon\C^n\to C$ and the covering by cylinder, $\pi_\cyl\colon\cyl\to C$  that extends to a covering $\cL M$ over $M$.  In section two we recall some facts about factors of automorphy  for   vector bundles on $C$ via the covering by $\C^n$. In section three, we study the flat vector bundles on $C$.    The pull back of the flat vector bundle $N_C$  to  the cylinder $\cyl$ is the normal bundle $N_{\cyl}$ of $\cyl$ in $\cL M$. We show that $N_{\cyl}$ is always the holomorphically trivial vector bundle $\cyl\times\C^d$. By ``vertical coordinates", we mean ``coordinates on $\C^d$", the normal component of the normal bundle $N_C$, while ``horizontal coordinates" mean the tangential components of  $N_C$.

Since $\cyl$ is a Stein manifold, a theorem of   Siu \cite{siu-stein} says that  a neighborhood of $\cyl$ in $\cL M$ is biholomorphic to a neighborhood of the zero section in its normal bundle, which is trivial as mentioned above.
We show that the holomorphic classification of neighborhoods $M$ of $C$ with flat $N_C$ is equivalent to 
the holomorphic classification of the family of the deck transformations of coverings $\cL M$ of $M$ in a neighborhood of $\cL C$. These deck transformations are
``higher-order" (in the vertical coordinates) perturbations $\tau_1,\dots, \tau_n$ of $\hat\tau_1,\dots, \hat\tau_n$, where the latter are the deck transformations of the covering of $\widetilde N_C$ over $N_C$. In order to find a biholomorphism between a neighborhood of $C$ in $M$ and a neighborhood of its
  zero 
 section in $N_C$, it is sufficient to find a biholomorphism that conjugates $\{\tau_1,\dots, \tau_n\}$ to $\{\hat\tau_1,\dots, \hat\tau_n\}$.

There are two useful features. First, since the fundamental group of $C$ is abelian, the deck transformations $\tau_1,\dots,\tau_n$ commute pairwise. Second, we can also introduce suitable coordinates on $\cyl$ so that the "horizontal" components of deck transformations have diagonal linear parts. In such a way the classification of neighborhoods of $C$ is reduced to a more attainable  classification of deck transformations. While the full theory for this classification is out the scope of this paper,  we study the case when $N_C$ admits  Hermitian transition functions. Since a Hermitian transition matrix must be locally constant, we call such an $N_C$ {\it Hermitian flat}.  The convergence proof for \rt{thm1.1} is given in section four. 
 It relies on a Newton rapid convergence scheme adapted to our situation  based on an appropriate Diophantine condition among the lattice and the normal bundle.
 At step $k$ of the iteration scheme, let  $\del_k$ be the error of the deck transformations  $\{\tau_1^{(k)},\dots, \tau_n^{(k)}$\} defined on domain $D^{(k)}$ to $\hat\tau_1,\dots, \hat\tau_n$ in suitable norms. By an appropriate transformation $\Phi^{(k)}$, we conjugate to a new set of deck transformations $\{\tau_1^{(k+1)},\dots, \tau_n^{(k+1)}\}$ of which the error to the linear ones is now $\del_{k+1}$ on a slightly smaller domain $D^{(k+1)}$. Using our Diophantine conditions, related to the lattice $\Lambda$ and the normal bundle, we show that the sequence $\Phi^{(k)}\circ \cdots \circ \Phi^{(1)}$ converges to a holomorphic transformation $\Phi$ on an open domain $D^{(\infty)}$ where we linearize $\{\tau_1,\dots,\tau_n\}$.

We now describe
closely related previous results.
Our work is motivated by work of Arnol'd and Ilyashnko-Pyartli. Our main theorem was proved in \cite{arnold-embed} when $C$ is an elliptic curve ($n=1$) and $N_C$ has rank one ($d=1$). Il'yashenko-Pyartli~\cite{ilyashenko-pyartly-embed} extended Arnol'd's result to the case when the torus is the product of elliptic curves  together with a normal bundle which is a direct sum of line bundles, while \rt{thm1.1} deals with general complex tori. We also assume that $N_C$ is {\it non-resonant}, a condition that is weaker
 than the non-resonant condition used 
by Il'yashenko-Pyartli. Our small-divisor condition is also weaker. Of course, the study of neighborhood of embedded compact complex manifolds has a long history. Also, see some recent work \cite{hwang-annals,koike-fourier,loray-moscou}.  We refer to \cite{MR4392029} for some references and a different approach to this range of questions.

{\bf Acknowledgments.} This work benefits from helpful discussions with Jean-Pierre Demailly. Part of work was finished when X.~G. was supported by CNRS and  UCA for a visiting position at UCA.

\setcounter{thm}{0}\setcounter{equation}{0}
	\section{Vector bundles on Tori and factors of automorphy}
In this section, we identify vector bundles on a complex torus with factors of automorphy. The latter gives us a useful alternative definition of vector bundles on a higher dimensional complex torus $C$ and the isomorphisms of two vector bundles. 	General references for line bundle on complex tori are \cite{birkenhake-lange,debarre-book,iena-05} and \cite[p.~307]{GH-book}.
	
   Let $\Lambda$ be a $2n$-dimensional lattice in $\mathbb{C}^n$.
  We may assume that  $\Lambda$ is defined by $2n$  vectors $e_1,\dots, e_n, e'_1,\dots,e'_n$  of $\mathbb{C}^n$, where $e_i=(0,\ldots, 0,1,0,\ldots, 0)$ with $1$ being at the $i$-th place, $e'_i=(e'_{i,1},\ldots,e'_{i,n})$ and the matrix
  $$\IM \tau:=(\IM{e'_{i,j}})_{1\leq i,j\leq n}=:({e''_{i,j}})_{1\leq i,j\leq n}$$ is invertible \cite[exerc. 2, p.~21]{birkenhake-lange}.
  The compact complex manifold $C:=\mathbb{C}^n/\Lambda$ is called an ($n$-dimensional) complex torus. Unless the lattice is equivalent to another one defined by a diagonal matrix $e'$, $C$ is not biholomorphic to a product of one-dimensional tori.
  Let $\pi: \C^n\rightarrow C$ be the universal cover of $C$. Its group $\Gamma$ of deck transformations consists of translations
  $$
 T_\la\colon z\to z+\la, \quad \la\in\Lambda.
  $$
  Note that $\Gamma$ is abelian and is isomorphic to $\Z^{2n}$. $\Gamma$ is also isomorphic to $\pi_1(C,0)$ since $\mathbb C^n$ is a universal covering of $C$.

 Next, we consider equivalence relations for holomorphic vectors bundles on $C$ and $\mathbb C^n$, following the realization proof of Theorem 3.2 in \cite{iena-05}.
 Let $E$ be a vector bundle of rank $d$ over $C$. The pull-back bundle $\pi^*E$ on $\mathbb C^n$ is  trivial and  has global coordinates $\hat\xi$. Let $\{U_j\}$ be an open covering of $C$ so that coordinates $\xi_j=(\xi_{j,1},\dots, \xi_{j,d})^t$ of $E$ are well-defined (injective) on $U_j$. Then we have
 \eq{hxixi}
 \hat\xi=h_j\xi_j(\pi),
 \eeq
  where $h_j$ is a non-singular holomorphic matrix on $\pi^{-1}(U_j)$. The transition functions $g_{kj}$ satisfy
 \eq{gkjh}
 g_{kj}(\pi)=h_k^{-1}h_j, \quad \text{on $\pi^{-1}(U_k)\cap\pi^{-1}(U_j)$}.
 \eeq
 For any $z,\la$, we know that both $\pi(z+\la)$ and  $\pi(z)$ are in the same $U_j$  for some $j$.  Then we have
$$
\hat\xi(z+\la)=h_j(z+\la)\xi_j(\pi(z))=h_j(z+\la)h_j(z)^{-1}\hat\xi(z).
$$
We can define
\eq{rholaz}
\rho(\la,z)=h_j(z+\la)h_j(z)^{-1}, \quad z\in\pi^{-1}(U_j)
\eeq
as the latter is independent of the choice of $j$
by 
\re{gkjh}.
%
Therefore,  
 \ga\label{global-coord}
 \hat\xi(\la+z)=\rho(\la,z)\hat\xi(z), \quad \rho(\la,z)\in GL(d,\mathbb C),\\
 \rho\colon\Lambda\times \mathbb C^n\to GL(d,\mathbb C).
 \end{gather}
Here $\rho$ is called a  {\it factor of automorphy}.
We can verify that
\eq{abelP}
\rho(\la+\mu,z)=\rho(\la,\mu+z)\rho(\mu,z).
\eeq
In particular, if all $\rho(\la,z)=\rho(\la)$ are independent of $z$, then 
 $\rho(\Lambda)$ is an abelian group.

The above construction from a vector bundle $E$ on $C$ to a factor of automorphy can be reversed. Namely, given
\re{global-coord}-\re{abelP}, 
define the vector bundle $E$
on $C$
as the quotient vector space of $\C^n\times\C^d$ via the equivalence relation
\eq{eqrel}
(z,\xi)\sim (z+\la,\rho(\la,z)\xi), \quad z\in\C^n,\ \xi\in\C^d,\ \la\in\Lambda.
\eeq
We denote the projection from the cylinder $\cyl:=\R^n/\Z^n+i\R^n=\C^n/\Z^n$ onto $C$ by $\pi_{\cyl}$.
Therefore, we can define $\pi_\cyl^*E$ on the cylinder $\cyl$ by the equivalence relation
\eq{eqrelS}
(z,\xi)\sim (z+\la,\rho(\la,z)\xi), \quad z\in\C^n,\ \xi\in\C^d,\ \la\in\Z^n.
\eeq

Of course,  global coordinates $\hat\xi$, $\rho$, and 
$g_{jk}$ are not uniquely determined by $E$. However, their equivalence classes are determined.  Two vector bundles $E,\tilde E$ are isomorphic if their corresponding transitions $g_{jk},\tilde g_{jk}$ satisfy $\tilde g_{jk}=h_j^{-1}g_{jk}h_k$ where $h_j$ are non-singular holomorphic matrices.   Replacing global coordinates $\hat\xi$ by $\nu\hat\xi$ where $\nu\colon\mathbb C^n\to GL(n,d)$ is holomorphic, we can verify that
\eq{nuLaz}
\nu(\la+z)\rho(\la, z)\nu(z)^{-1}=:\tilde\rho(\la,z)
\eeq
is also a factor of automorphy.   Define two factors of automorphy  $\rho,\tilde\rho$ to be {\it equivalent} if \re{nuLaz} holds. Therefore,  the classification of holomorphic vector bundles is identified with the classification of factors of automorphy.

\setcounter{thm}{0}\setcounter{equation}{0}
\section{Flat vector bundles}
In this section, we will show that the pull-back of a flat vector bundle $E$ on $C$ to the cylinder  $\cyl=\C^n/\Z^n$ is always trivial.

 When  $E$ is flat, we can choose global coordinates as follows. We know that  $\pi^*E$ is also flat and   we can choose its global flat basis, or global flat coordinates  $\hat\xi$ by using analytic continuation on $\C^n$ and pulling back flat local coordinates of $E$. In other words,  in \re{hxixi} $h_j$ are locally constants while $\xi_j$ are locally flat coordinates.  Then $\rho(\la,z)$ depend only on $\la$, in which case we write $\rho(\la)$ for $\rho(\la,z)$. As remarked above, $\rho(\Lambda)$ is abelian.
When $E$ is unitary (flat), by the same reasoning  $\pi^*E$ is unitary and we can choose
$h_j$ and $\rho(\la)$ to be unitary.

  A $d\times d$ Jordan block $J_d(\la)$  is a matrix of the form $\lambda I_d+N_d$, where $I_d$ is the $d\times d$ identity matrix and $N_d$ is the $d\times d$ matrix with all entries being $0$, except all
 the $(i,i+1)$-th entries being $1$. A  matrix $T$  commutes with $J$ if and only if
 $$
T=T_d(a):=a_0I_d+\sum_{i>0} a_i N_d^i.
 $$
 Note that $N_d^i=0$ for $i\geq d$. Following \cite[p.~218]{gantmacher1},
 we call the above $T$ as well as the  following two types of matrices, {\it regular upper triangular} matrices  w.r.t. $J$:
   $$
 A=(0,T_d(a)), \quad\text{or} \quad B=\binom{T_{d'}(a')}{0},
 $$
  where $0$ denotes in $A$ (resp. $B$) a $0$ matrix of $d$ rows (resp. $d'$ columns).
 Given a Jordan matrix
 $$
 \tilde J=\diag(J_{d_1}(\la_1),\dots, J_{d_k}(\la_k)). $$
 the matrices that commute with $\tilde J$ are precisely the block matrices
 $$
X= (X_{\all\beta})_{m\times m}
 $$
 where $X_{\all\beta}=0$ if $\la_\all\neq\la_\beta$, while $X_{\all\beta}$ is a regular upper triangular $(d_\all\times d_\beta)$ matrix if $\la_\all=\la_\beta$.
 Such a matrix $X$ is said to be a {\it regular upper triangular matrix }
 w.r.t. $J$

From the structure of matrices commuting with a Jordan matrix, we can verify the following two results.
\begin{prop}Let $A_1,\dots, A_m$ be $2\times 2$ matrices commuting pairwise. Then there is a non-singular matrix $S$ such that all $S^{-1}A_jS$
are Jordan matrices.
\end{prop}
\begin{exmp} The $3\times 3$ matrices $\la I_3+N_3,
\mu I_3+N_3^2$ commute,  but they cannot be transformed into the Jordan normal forms simultaneously.
\end{exmp}
The following results on logarithms are likely classical. However, we cannot find a reference. Therefore, we give  proofs emphasizing  commutativity of logarithms  of matrices.

We start with the following.
\begin{lemma}\label{up-t}
Let $A_1,\dots, A_m$ be pairwise commuting matrices. Then there is a non-singular matrix $S$ such that
$$
S^{-1}A_jS=(\hat A_{\all\beta})_{1\leq\all,\beta\leq s}=:\hat A_j, \quad 1\leq j\leq m
$$
where  $\hat A_1$ is a Jordan matrix and all $\hat A_j$ are
upper triangular matrices.
\end{lemma}
 \begin{proof} We may assume that  $A_1$ is a Jordan matrix $J=\diag(J_{d_1}(\la_1),\dots, J_{d_s}(\la_s))$. Then
 $$
 A_j=(X^j_{\all\beta})_{1\leq\all,\beta\leq s}
 $$
 are regular w.r.t $J$. Note that pairwise commuting non-singular matrices have  non-trivial common eigenspaces.  The eigenspace of $J$ are spanned $e_{d_1'}, \dots, e_{d'_s}$ with $d'_1=1$ and $d_j'=d_1+\cdots d_{j-1}+1$.
 To simplify the indices, we may assume that $e_1$ is an eigenvector of all $A_j$. Then the first column of $A_j$ is $a_je_1$ with $a_j\neq0$.  The new matrices $\tilde A_j$,  obtained by removing all first rows and first columns,
  still commute pairwise.  In particular all $\tilde A_j$ are regular to the new Jordan matrix $\tilde J=\tilde A_1$. By induction on $d$,   we can find  a non-singular matrix  $\tilde S$ which is regular to $\tilde J$ so that all $\tilde S^{-1}\tilde A_j\tilde S$ are
upper-triangular.  Let $S=\diag(1,\tilde S)$. Now  $S^{-1}=\diag(1,\tilde S^{-1})$. We can check that
all $\hat A_j:=S^{-1}A_jS$ are upper triangular. Then $ \hat A_1$, $J$ have the same entries, with only one possible exception
$$\hat A_{1;12}=s_{11}J_{12}.
$$
If $\hat A_{1;12}\neq 0$, dilating the first coordinate can transform $\hat A_{1}$ into the original $J$, while  $\hat A_j$ remain upper-triangular. If $s_{12}J_{12}=0$, then $J_{12}$ must be $0$, i.e. $\hat A_1=J$,  because one cannot transform a Jordan matrix, $A_1=J$, into a new Jordan matrix,
$\hat A_1$, by reducing  an entry $1$ to $0$ and keeping other entries unchanged.
 \end{proof}
 The above simultaneous normalization of  upper-triangular matrices allows us to define the logarithms. The construction of logarithms of non-singular matrices can be found in \cite[p.~239]{gantmacher1}. Here we need to find a definition that is suitable to determine the commutativity of the logarithm of pairwise commuting non-singular matrices.

Recall that for a  $d\times d$ matrix $A$, the generalized eigenspace $E_\la (A)$ with eigenvalue $\la$ is the kernel of $(A-\la I)^d$, while $\C^d$
is the direct sum of all $E_\la (A)$. A
matrix $B$
that commutes with $A$ leaves each $E_\la (A)$ invariant, i.e. $B(E_{\la}(A))\subset E_{\la}(A)$. Thus if $A_1,\dots, A_m$ commute pairwise, we can decompose $\C^d$ as a direct sum of liner subspaces $V_j$ such that each $V_j$ is invariant by $A_i$ and admits exactly one eigenvalue of $A_i$. Thus to define $\ln A_j$, we will assume that each $A_j$ has a single eigenvalue on $\C^d$ if we wish.

 Given a non-singular matrix $A$, a logarithm of $A$ is a matrix $\ln A$  satisfying
 $$
 e^{\ln A}=A
 $$
 where the exponential matrix $e^B=\sum\frac{B^n}{n!}$ is always well-defined.  However,  $\ln A$
is not  unique.


  For a non-singular upper triangular matrix
 \begin{equation}\label{defLn0}
 A=\la I_d+ a, \quad a:=(a_{ij})_{1\leq i,j\leq d}, a_{ij=0}\quad  \forall i\geq j,
 \end{equation}
 we have $a^k=0$ for $k\geq d$.
  Using the identity $e^{B+C}=e^{B}e^C$ for two commuting matrices $B,C$, we  see that $e^{\ln A}=A$ for
 \begin{equation}\label{defLn0+}
 \ln A:=(\ln \la)I_d-\sum_{k>0}\f{(-\la^{-1} a)^k}{k}
 \end{equation}
with $0\leq \IM  \ln(.) <2\pi$.
For a non-singular Jordan matrix, we can define
$$
\ln\diag(J_{d_1}(\la_1),\dots, J_{d_m}(\la_m))=\diag(\ln J_{d_1}(\la_1),\dots, \ln J_{d_m}(\la_m)).$$
 Note that $\ln\la_\all=\ln\la_\beta$ if and only if $\la_\all=\la_\beta$.
 Since matrices that commute with a fixed matrix is closed under multiplication by a scalar, addition and multiplication.
 It is thus clear that if $A$ is an upper triangular matrix that is regular to a non-singular Jordan matrix $J=\diag(J_{d_1}(\la_1),\dots, J_{d_s}(\la_s))$, then $\ln A$ remains regular to $J$. Equivalently and more importantly,  $\ln A$ is regular to the Jordan normal form of $\ln J$, which is
 $$
 \diag(J_{d_1}(\ln\la_1),\dots, J_{d_s}(\ln\la_s)).
 $$

 \begin{prop}\label{defLnA}
 Let $A_1,\dots, A_m$ be pairwise commuting 
 $d\times d$ matrices. Then there is a non-singular matrix $S$ such that
 $\hat A_j:=S^{-1}A_jS$ are block diagonal matrices
  of the form \begin{equation}\label{hataj}
 \hat A_j=\diag(\hat A_{j,d_1},\dots, \hat A_{j,d_k})
 \end{equation}
  where all $\hat A_{j,d_i}$
  are   upper triangular $d_i\times d_i$ matrices, and each $\hat A_{j,d_i}$ has only one eigenvalue $\la_{j,i}$. Assume further that all $A_j$ are non-singular. Then
  \begin{equation}\label{defLn1}
  \ln A_j:=S\diag(\ln \hat A_{j,d_1}, \dots, \ln \hat A_{j,d_k})S^{-1}, \quad 1\leq j\leq m
  \end{equation}
  commute pairwise and $e^{\ln A_j}=A_j$, where $\ln\hat A_{j,d_i}$ are defined by  \rea{defLn0}-\rea{defLn0+}.
 \end{prop}
 \begin{proof}Note that for pairwise commuting   matrices $A_1,\dots, A_m$, we have a decomposition
 $
  \C^d=\bigoplus_{i=1}^tV_i
 $
 where each $A_j$ preserves $V_i$ and
 has only one eigenvalue $\la_{j,i}$. Their restrictions of $A_1,\dots, A_m$
  on $V_i$ remain commutative pairwise.
 Let $\dim V_i=d_i, d_0'=0, d_{i+1}'-d_i'=d_i$.
 By \rl{up-t}, we can find a basis
 $$e^*_{d'_{i-1}+1}, \dots, e^*_{d_{i-1}'+d_{i}}$$
 for $V_i$ such that $A_1|_{V_i}, \dots, A_m|_{V_i}$ are upper triangular matrices. Using the new basis $e_1^*,\dots, e_d^*$ we can find the matrix $S$ for the decomposition \re{hataj}. Therefore,
  $ \hat A_{1, d_i},\dots, \hat A_{m,d_i}$
  commute pairwise.

Assume now that all $\la_{j,i}$ are non-zero.
 It remains to show that $\ln \hat A_{1, d_i},\dots, \ln \hat A_{m,d_i}$ commute pairwise.
 Write
 $$
 \hat A_{j,i}=\la_{j,i}I_{d_i}+W_{j,i}
 $$
 where $W_{j,i}$ are upper triangular  matrices and  $W_{j,i}^{d_i}=0$.
 By a straightforward computation, we have
 $$
 [W_{j,i}, W_{j',i}]=[\hat A_{j,i},\hat A_{j',i}]=0.
 $$
 Therefore, $W_{j,i}^k$ commutes with $W_{j',i}^\ell$ for any $k,\ell$. Consequently,  the finite sum
 $$
 \ln \hat A_{j,i}=\ln\hat\la_{j,i}I_{d_i}-\sum_{k>0}\frac{(-\la_{j,i}^{-1}W_{j,i})^k}{k}
 $$
 commutes with $\ln\hat A_{j',i}$. Therefore,  $\ln A_1,\dots, \ln A_m$ in \re{defLn1}
 commute pairwise.
  \end{proof}

    \begin{lemma}\label{muit-flows}
 Let $A_1,\dots, A_n$ be   non-singular  upper triangular   $d\times d$ matrices.
 Suppose that $A_1,\dots, A_n$ commute pairwise.
 There exists a linear mapping $w\to\tilde v^z(w):=v(z)w$ in  $\C^d$,
  entire in $z\in\C^n$  such that 
 $v(z)\in GL_d(\C)$, $v(0)=Id$ and $v(e_j)=A_j$ for $j=1,\ldots,n$.  Furthermore,  $v({z+z'})=v(z) v({z'})$ for all $z,z'\in\C^n$.
 \end{lemma}
 \proof{By \rp{defLnA}, we  define pairwise commuting matrices $\ln A_1,\dots, \ln A_n$ such that $e^{\ln A_j}=A_j$.
  Then the Lie brackets of the vector fields for $\dot w=\ln A_jw, j=1,\dots,n$ vanish and their flows $\var_j^t(w)$  commute pairwise.
 Note that
 $$
 \var_j^0(w)=w, \quad
 \var_j^1(w)=e^{\ln A_j}w=A_jw.
 $$
 Define
 $$
 \tilde v^z(w)=\var_1^{z_1}\cdots\var_n^{z_n}(w).
 $$
  We conclude $\tilde v^{z}\tilde v^{z'}(w)=\tilde v^{z+z'}(w)$, that is $v(z+z')=v(z)v({z'})$.
\qedhere }


By \rp{defLnA}, we define $\ln\rho(e_1), \dots, \ln\rho(e_{2n})$ and they commute pairwise.
We now define $\ln\rho(\la)$ for all $\la=\sum_{j=1}^{2n}m_je_j\in \Lambda$ as follows
$$
\ln\rho\left(\sum_{j=1}^{2n}m_je_j\right):=\sum_{j=1}^{2n} m_j\ln\rho(e_j).
$$
Thus the matrices $\ln\rho(\la)$ for $\la\in\Lambda$ commute pairwise.

\begin{prop}\label{pi_SEtrivial}
Let $E$ be a flat vector bundle on $C$.
  Then  $\pi_{\cyl}^*E$ admits a factor of automorphy $\rho$ satisfying $\rho(e_j)=Id$ for $j=1,\dots, n$; in particularly, $\pi_\cyl^*E$ is  holomorphically trivial.
\end{prop}
\proof{  Let $d$ be the rank of $E$.
Let $A_j=\rho(e_j)^{-1}$ for $j=1,\dots, n$. With $v(e_j)=A_j$ and $v(0)=Id_d$, we first see that
$$
\tilde\rho(\la,z):=v(z+\la)\rho(\la)v(z)^{-1}
$$
satisfies $\tilde\rho(e_j,0)=Id_d$.  We want to show that $\tilde\rho(\la,z)$ depends only on $\la$. Fix  $\la=\sum_{j=1}^{2n}m_je_j\in \Lambda$. By definition,
the matrix $\ln\rho(\la)$ commutes with each $\ln\rho(e_j)$, $j=1,\ldots, 2n$. Thus the flow $\var_\la^t$ of $\dot w=
\ln\rho(\la)w$ commutes with the flows of
$\dot w=
\ln\rho(e_j)w$, $j=1,\ldots, 2n$.  As in the proof of the previous lemma, we know that $\var_\la^t(w)$ is linear in $w$ and entire in $t\in\C$. For $z\in \C^n$ and $w\in \C^d$, let $\tilde v^z(w)$ be as defined in the previous lemma.   Thus we have
$$
\var_\la^t \tilde v^z(w)=\tilde v^z\var_\la^t(w).
$$
Taking derivatives in $w$ and plugging in $t=1$, we get $$
\exp(\ln\rho(\lambda))v(z)=v(z)\exp(\ln\rho(\lambda)).
$$ Since $\exp(\ln\rho(\lambda))=\rho(\lambda)$, we have  $\rho(\la)v(z)=v(z)\rho(\la)$ for all $\la\in\Lambda$, $z\in \C^n$.
Considering $z=z_1e_1+\cdots +z_ne_n\in \Lambda$, we have $v(z+\la)=v(z)v(\la)$. Hence,    $\tilde\rho(\la, z)$ is independent of $z$.

We have achieved $\tilde\rho(\la)=\nu(z+\la)\rho(z)\nu(z)^{-1}$ and $\tilde\rho(e_j)=Id_d$ for $j=1,\dots, n$. Therefore, $\pi|_\cyl^*E$ is trivial, by the equivalence relation \re{eqrelS}.
\qedhere}

It is known that there are Stein manifolds with non-trivial vector bundle \cite{forster-rammspott}. Furthermore, we conclude the section by emphasizing that the triviality $\pi^*|_\cyl E$ relies on the extra assumption that it is a pull-back bundle. The following result is likely known, but we include a short proof for completeness.
\begin{prop}The set of holomorphic equivalence classes of flat holomorphic line bundles on $\cyl$ can be identified with $H^1(\cyl,\mathbb C^*)$. The latter is non-trivial.
\end{prop}
\begin{proof} Each element $\{c_{jk}\}$ in $H^1(\cyl,\mathbb C^*)$ is clearly an element in $H^1(\cyl,\mathcal O^*)$.
We want to show that if $\{c_{jk}\},\{\tilde c_{jk}\}\in H^1(\cyl,\mathbb C^*)$ represent the same element in   $H^1(\cyl, \mathcal O^*)$, then they are also the same in $H^1(\cyl,\mathbb C^*)$. Indeed, we can cover $\cyl$ by convex open sets $U_1,U_2,U_3,U_4$ such that $U_1\cap U_2\cap U_3\cap U_3$ is non empty. Thus $\{U_i\}$ is a Leray covering. If $\tilde c_{jk}=h_jc_{jk}h_k^{-1}$, where each $h_j$ is a non-vanishing holomorphic function on $U_j$. Take $p$ in all $U_j$. We get $\tilde c_{jk}=c_jc_{jk}c_k^{-1}$ for $c_j=h_j(p)$.

  Note that
  $H^1(\cyl,\mathbb C^*)$ is non-trivial.
  Otherwise,   the  exact sequences $0\to\mathbb Z\to\mathbb C\to\mathbb C^*\to0$ and
  $0\to  H^0(\cyl,\mathbb Z)\to  H^0(\cyl,\mathbb C)\to  H^0(\cyl,\mathbb C^*)\to 0$ would imply that
  $H^1(\cyl, \Z) \cong H^1(\cyl,\C)$, a contradiction.
 \end{proof}
\setcounter{thm}{0}\setcounter{equation}{0}

%

\setcounter{thm}{0}\setcounter{equation}{0}

\section{Equivalence  of neighborhoods and commuting deck transformations}

 In this section, we will discuss how the classification of neighhorhoods $U$ of a compact complex manifold $C$ is related to the classification of deck transformations of a holomorphic covering $\tilde U\to U$, where $U,\tilde U$ are chosen carefully and $\tilde U$ contain $C^*$ that covers $C$. When $C^*$ is additionally Stein, we can choose $\tilde U$ to be a neighborhood of $C^*$ in its normal bundle $N_{C^*}(\tilde U)$ 
  by applying a result of Siu.  After  preliminary results in \rl{deck-tran} and \rl{CinN}, we will return to our previous study where $C$ is a complex torus, and its covering of $C$ is the Stein manifold  $C^*=\tilde C$, and $N_C(M)$ is Hermitian flat. We then prove the main result of this paper by using a KAM rapid iteration scheme.

Let us start with $\iota: C\hookrightarrow M$, a holomorphic embedding of a compact complex manifold $C$. We shall still denote $\iota(C)$ by $C$.  Let $U$ be a neighborhood of $C$ in $M$ such that $U$ admits a smooth, possibly non-holomorphic,  \emph{deformation or strong retract} $C$~\ci[p.~361]{MR3728284}; namely there is a smooth mapping $R\colon U\times [0,1]\to U$ such that $R(\cdot,0)=Id$ on $U$, $R(\cdot,t)=Id$ on $C$, and $R(\cdot,1)(U)=C$. Thus, $\pi_1(U,x_0)=\pi_1(C,x_0)$ for $x_0\in C$ (see~\ci[p.~361]{MR3728284}).   When $M$ is $N_C$, we can find a {\it holomorphic}  deformation retraction from a suitable neighborhood of its the zero section onto $C$,   by using  a  Hermitian metric on $N_C$.

Let $X$ be a complex manifold and  $\cL X$ be a universal covering of $X$. Then the group of deck transformations of the covering is identified with $\pi_1(X, x_0)$. The set of equivalence classes of coverings of $X$ is identified with the conjugate classes of subgroups of $\pi_1(X, x_0)$; see~\ci[Thm.~79.4, p.~492]{MR3728284}. Furthermore, $\pi_1(X, x_0)$ acts transitively and freely on each fiber of the covering and $X$ is the quotient of $\cL X$ by $\pi_1(X, x_0)$.
\begin{lemma}\label{deck-tran}Let $C$ be a compact complex manifold. Let $\pi\colon C^*\to C$ be a holomorphic covering and $\pi(x_0^*)=x_0$. Suppose that $(M,C)$ (resp. $(M', C)$) is a holomorphic neighborhood of $C$. There is a neighborhood $U$ in $M$ (resp. $U'$ in $M'$) of $C$ 
 and  a holomorphic neighborhood  $\tilde U$ $($resp. $\widetilde{U'})$ of $C^*$ such that $p\colon\tilde U\to U$ $($resp $p\colon\widetilde{U'}\to U')$ is an extended covering of the covering $\pi\colon C^*\to C$ and $C$
  $($resp. $C^*)$   is a smooth strong retract of $U,U'$ $($resp. $\tilde U,\widetilde{U'})$. Consequently, $$
  \pi_1(\tilde U,x_0^*))=\pi_1(C^*,x_0^*), \quad \pi_1(U,x_0)=\pi_1(C,x_0).
$$
Suppose that $(M,C)$ is biholomorphic to $(M',C)$. Then $U,U', \tilde U,\widetilde{U'}$
can be so chosen that there is covering transformation sending $\tilde U$ onto $\widetilde{U'}$ and fixing $C^*$ pointwise. Conversely, if there is a convering transformation sending $\tilde U$ onto $\widetilde{U'}$ fixing $C^*$ pointwise, then $(U,C),(U',C)$ are holomorphically equivalent.
\end{lemma}
\begin{proof}
 Since $\pi\colon C^*\to C$ is a covering map,  
according to \cite[Thm.~4.9]{vick-book}, it extends to a covering map $p:\tilde U \to U$  such that $\tilde U$ contains $C^*$ and $p|_{C^*}=\pi$. Suppose that $R$ is a strong retraction of $U$ onto $C$. We can lift $R(z,\cdot)\colon[0,1]\to U$ to a continuous mapping $\tilde R(\tilde z,\cdot)\colon[0,1] \to \tilde U$ such that $\tilde R(\tilde z,0)=\tilde z$ and $p\tilde R(\tilde z,.)=R(p(\tilde z),.)$ for all $\tilde z\in\tilde U$. One can verify that $\tilde R$  is a strong retraction  of $\tilde U$ onto $C^*$.

Suppose that a biholomorphic map $f$ sends  $(M,C)$ onto $(M',C)$ fixing $C$ pointwise. We may assume that $f$ is a biholomorphic mapping from $U$ onto $U'$. Then we can lift the mapping $f\pi\colon \tilde U\to { U'}$ to obtain a desired covering biholomorphism $F$, since $(f\pi)_*\pi_1(\tilde U,x_0^*)=\pi_*\pi_1(C^*,x_0^*)$. Conversely, a covering biholomorphism from $\tilde{U}$ onto $\tilde{ U'}$ fixing $C^*$ pointwise clearly induces a biholomorphism from $U$ onto $U'$ fixing $C$ pointwise.
\end{proof}

With the covering, we can identity $(M,C)$ with $(\tilde M, C^*)/{\,\sim}$ where $\tilde p\sim p$ if and only if $p,\tilde p$ are in the same stack of the covering $\pi$.

Applying the above to $(N_C,C)$ and a covering $\pi|_{C^*}\colon C^*\to C$, we have a covering $\hat\pi\colon \widetilde{N_C}\to N_C$ such that
$$
C^*\subset\widetilde{N_C}, \quad \pi_1(\widetilde{N_C},x_0^*)=\pi_1(C^*,x_0^*), \quad \pi_1(N_C,x_0)=\pi_1(C,x_0).
$$

To simplify the notation, we denote $U,\tilde U$ by $M,\tilde M$ respectively.
Thus we have commuting diagrams for the coverings:
$$
\begin{matrix}
	\widetilde{N_C} & \hookleftarrow& C^* & \hookrightarrow & {\tilde M}\\
	\hat \pi\downarrow & 	&\pi_{C^*}\downarrow& & p\downarrow \\
	N_C& \hookleftarrow &	C & \hookrightarrow & M
\end{matrix}.
$$
The set of deck transformations of $p$ (resp. $\hat\pi$) will be denoted by $\{\tau_1,\dots, \tau_n\}$ (resp. $\{\hat\tau_1,\dots, \hat\tau_n\}$).
If $\pi\colon \tilde U\to U$ is a covering map, $Deck(\tilde U)$ denotes the set of deck transformations.

\begin{lemma}\label{CinN}
Let $C, C^*, M$ be as in \rla{deck-tran}. Suppose that $C^*$ is a Stein manifold. Let $\om_0^*$ be an open set of $\widetilde{N_C}$ such that $\hat\pi(\om^*)$ contains $C$. Then there is an open subset $\om^*$ of $\om^*_0$ such that $\hat\pi\om^*$ contains $C$ and  $(M,C)$ is holomorphically equivalent to the quotient space of $\om^*$ by $Deck( \widetilde{N_C})$.
\end{lemma}
\begin{proof}
 By a result of Siu \cite[Cor.~1]{siu-stein}, we find a biholomorphism $L$ from a holomorphic strong retraction neighborhood of $C^*$ in $\tilde M$, still denoted by $\tilde M$ into  $N_{C^*}(\tilde M)$ and  a biholomorphism $L'$ from a   strong retraction neighborhood of $C^*$ in $\widetilde{N_C}$, still denoted by  $\widetilde{ N_C}$ into  $N_{C^*}( \widetilde{N_C)}$. Furthermore, $L,L'$ fix $C^*$ pointwise.
  We have
 \aln
 p_*\pi_1(N_{C^*}(\tilde M),x_0^*)&=p_*L^{-1}_*\pi_1(\tilde M,x_0^*)=\pi_1(C,x_0)=\pi_1(L^{-1}\tilde M,x_0^*)\\
 &=\hat\pi_*(L')^{-1}_*(\pi_1(\widetilde{N_C},x_0^*)).
 \end{align*}
 Both $\hat\pi\circ L'^{-1}\colon N_{C^*}( \widetilde{N_C)})\to M$ and $p\circ L^{-1}\colon N_{C^*}( \tilde M)\to M$ are  coverings   and the above identifications show that the lifts of the two coverings yield a biholomorphism between neighborhoods of $C^*$ in $\tilde M$ and $N_{C^*}(\widetilde{N_C})$ fixing $C^*$ pointwise.
 \end{proof}

Here, $z=(z_1,\ldots,z_n)$ belongs to the fundamental domain
$$
\om_0=\left\{\sum_{j=1}^{2n} t_je_j\in\C^n\colon t\in[0,1)^{2n}\right\},\quad e_{n+i}:=\tau_i,\quad i=1,\ldots,n.
$$
Thus $h$ belongs to the fundamental domain $\Om_0$ defined
$$
{\Om 
}_0:=\{(e^{2\pi i\zeta_1}, \dots e^{2\pi i\zeta_n})\colon\zeta\in \om_0\},\quad {
\Om 
}_0^+=\{(|z_1|,\dots, |z_n|)\colon z\in {\Om 
}_0\}.
$$
Thus ${\Om 
}_0$ is a Reinhardt domain, being $\{(\nu_1R_1,\dots, \nu_nR_n)\colon |\nu_j|=1\colon R\in \Om 
_0^+\}$.
We have
$$
{\Om 
}_0^+=\left\{(e^{-2\pi R_1},\dots, e^{-2\pi R_n})\colon R=\sum_{i=1}^n t_i\IM\tau_i,t\in[0,1)^n\right\}.
$$
For $\epsilon>0$,  define a (Reinhardt) neighborhood ${\Om 
}_\epsilon$ of $\ov{{\Om 
}_0}$ by
\begin{gather*}
\om_\epsilon:=\left\{\sum_{j=1}^{2n} t_je_j\colon t\in[0,1)^n\times (-\epsilon,1+\epsilon)^{n}\right\},
\quad {\Om 
}_\epsilon:=\{(e^{2\pi i\zeta_1}, \dots e^{2\pi i\zeta_n})\colon\zeta\in \om_\epsilon\}.
\end{gather*}
For any $\ell$-tuple of indices in $\{n+1,\ldots,2n\}$, we set 
\ga{}
\om_\e^{j_1\dots j_\ell}:=
\left\{\sum_{j=1}^{2n} t_je_j\colon t_{k}\in(-\e,\e), k-n\in \{j_1,\dots, j_\ell\};   t_{k}\in(-\e,1+\e), k-n\not\in \{j_1,\dots, j_\ell\}
  \right\},
 \\
\tilde \om_\e^{j_1\dots j_\ell}:=
\left\{\sum_{j=1}^{2n} t_je_j\colon
t_{k}-1\in(-\e,\e),  k-n\in \{j_1,\dots, j_\ell\}; t_{k}\in(-\e,\e),  k-n\not\in \{j_1,\dots, j_\ell\}
 \right\}.
\end{gather}
Note that $\tilde \om_\e^{j_1\dots j_\ell}$ and $ \om_\e^{j_1\dots j_\ell}$ are subset of $\om_\e$, and
   $\om_\e^{1\dots n}=
\{\sum_{j=1}^{2n} t_je_j\in\om_\epsilon\colon t\in[0,1)^n\times(-\e,\e)^n\}$.  Then
\begin{gather}
\Om_\e^{j_1\dots j_\ell}:=\Om_\epsilon\cap\bigcap_{k=1}^\ell T_{j_k}^{-1}\Om_\epsilon=\{(e^{2\pi i\zeta_1}, \dots e^{2\pi i\zeta_n})\colon\zeta\in \om^{j_1\dots j_\ell}
_\epsilon\},
\\
\tilde\Om_\e^{j_1\dots j_\ell}:=\Om_\epsilon\cap\bigcap_{k=1}^\ell T_{j_k}\Om_\epsilon=\{(e^{2\pi i\zeta_1}, \dots e^{2\pi i\zeta_n})\colon\zeta\in \tilde\om^{j_1\dots j_\ell}\}
\end{gather}
are connected non-empty Reinhardt domains. Moreover, $\Om_0^{1\cdots n}:=\cap_{\e>0} \Om_0^{1\cdots n}$ and $ \tilde\Om_0^{1\cdots n}=\cap_{\e>0}\Om_\e^{1\cdots n} $ are diffeomorphic to the real torus $(S^1)^n$.
We remark that
$
T_iT_j$ maps $\Om^{ij}_\e$ into $\Om_\e$ for $i\neq j$, while $T_i\circ T_i$ does not map $\Om_{\e,r}^{ii}$ into   $\Om_\e$ .


With $\Del_r=\{z\in\C\colon|z|<r\}$, we also define 
\beq\label{domains}
\om_{\epsilon,r}:= \omega_\epsilon\times  \Delta^d_r, 
\quad \Omega_{\epsilon,r}:=\Om 
_\epsilon\times  \Delta^d_r. 
\eeq
Throughout the paper, a mapping $(z',v')=\psi^0(z,v)$ from $\om_{\e,r}$ into $\C^{n+d}$ that commutes with $z_j\to z_j+1$ for $j=1,\dots, n$ will be identified with a well-defined mapping $(h',v')=\psi(h,v)$ from $\Om_{\e,r}$ into $\C^{n+d}$, where $z,h$ and $z',h'$ are related   as in \re{hTj}. A function on $\om_{\e,r}$ that has period $1$ in all $z_j$ is identified with a function on $\Om_{\e,r}$. We shall use these identifications as we wish.
\begin{prop}Let $C$ be the complex torus and $\pi_{\cyl}\colon\cyl=\C^n/{\Z^n}\to C$ be the covering.
Let $(M,C)$ be a neighborhood of $C$. Assume that $N_C$ is flat.
\bppp
\item Then one can take $\om_{\e_0,r_0}=\om_{\e_0}\times \Delta_{r_0}^d$ such that $(M,C)$ is biholomorphic to the quotient of $\om_{\e_0,r_0}$ by  $\tau^0_1,\dots,\tau^0_n$. Let $\tau_j$ be the mapping defined on $\Om_{\e_0,r_0}$ corresponding to $\tau_j^0$.
Then $\tau_1,\dots, \tau_n$ commute pairwise   wherever they are defined, i.e.
$$
\tau_i\tau_j(h,v)=\tau_j\tau_i(h,v)\quad  \forall  i\neq j
$$
for $(h,v)\in \Om_{\e_0, r_0}\cap \tau_i^{-1}\Om_{\e_0,r_0}\cap \tau_j^{-1}\Om_{\e_0,r_0}$.
\item
 Let $(\tilde M,C)$ be another  such neighborhood having the corresponding generators $\tilde\tau_1,\dots,\tilde\tau_n$  of deck transformations defined on $\Om_{\tilde\e_0,\tilde r_0}$. Then $(M,C)$ and $(\tilde M,C)$ are holomorphically equivalent if and only if there is a  biholomorphic mapping $F$  from $\Om_{\e,r}$ into $\Om_{\tilde \e,\tilde r}$ for some positive $\e,r,\tilde e,\tilde r$  such that
$$
F \tilde\tau_j(h,v)=\tau_jF(h,v),\ j=1,\ldots, n,
$$
wherever   both sides are defined, i.e. $(h,v)\in \Om_{\tilde\e,\tilde r}\cap \tilde\tau_j^{-1}\Om_{\e,r}\cap \Om_{\e,r}\cap F^{-1}\Om_{\e,r}.$
\eppp
\end{prop}
\begin{proof}
We now apply \rl{CinN}, in which $C^*$ is replaced by
  $\cyl=\R^n/\Z^n+i\R^n$ is a Stein manifold.
  Assume that $N_C$ is flat. Then according to \rp{pi_SEtrivial},  $N_{\cyl}( \widetilde{N_C)}=N_{\cyl}(\tilde M)=\pi_{\cyl}^*(N_C)$  is the trivial vector bundle  $\cyl\times \C^d$ with coordinates $(h,v)$, while $\cyl\times\{0\}$ is defined by $v=0$.
\end{proof}

Set
\gan
{\mathcal P}_\epsilon^+:=\left\{ \sum_{i=1}^n t_i\IM\tau_i\colon t\in (-\epsilon,1+\epsilon)^{n}\right\},\\
{\Om 
}_\epsilon^+:=\left\{(e^{-2\pi R_1},\dots, e^{-2\pi R_n})\colon R
\in {\mathcal P}_\epsilon^+
\right\}.
\end{gather*}
Note that  ${\mathcal P}_\epsilon^+, {\mathcal P}_0^+$ are $n$-dimensional parallelotopes,
 and $\Om_\e^+$ contains  $(1,\ldots,1)$, the image of $0\in{\mathcal P}_\epsilon^+$, corresponding to the real torus $(S^1)^n$.

 Since $\Om 
 _\epsilon$ is   Reinhardt, we have
 \eq{boundary-of-R}
 \Om 
 _\epsilon\supset\Om 
 _\epsilon^+,\quad (\pd \Om 
 _\e)^+=\pd(\Om 
 _e^+), \quad (\pd \Om 
 _\e)^+:=\{(|h_1|,\dots,|h_n|)\colon h\in \pd \Om 
 _\e\}.
 \eeq
 We now apply the above general results to the case where $C$ is a complex torus,  $C^*=\tilde C$ and $N_C(M)$ is Hermitian flat.

 As in \cite{ilyashenko-pyartly-embed}, the deck transformations of $(\tilde N_C,\cyl)$ are   generated by $n$ biholomorphisms $\hat\tau_1,\dots, \hat\tau_n$ that preserve $\cyl$.\begin{gather}
 	\hat \tau_j(
 	h,v)
 	=
 	(T_jh,  M_jv),\quad M_j:=\diag(\mu_{j,1},\dots, \mu_{j,d})\label{tauhat}
 \end{gather}
 with $h,T_j$ being defined by~:
  \eq{hTj} h=(e^{2\pi i z_1}, \cdots, e^{2\pi iz_n}), \quad
 T_j:= \diag(\la_{j,1},\dots, \la_{j,n}),\quad \la_{j,k}:=e^{2\pi i\tau_{jk}}.
 \eeq
Here the invertible $(d\times d)$-matrix $M_j$ is the factor of automorphy $\rho(e_{n+j})$ of $\pi_{\cyl}^*N_C$.

Recall that each deck transformation
 $\tau^0_j(z,v)$, $j=1,\ldots, n$, is a holomorphic map defined on $\overline \om_{\epsilon,r}$.
   In coordinates $(h,v)$, $\tau_j^0$ becomes $\tau_j$ defined on $\overline\Om_{\e,r}$.
      Since $T_CM$ splits, it is a higher-order perturbation of 
      $\hat\tau_j$ 
      with $$
      \tau_j(h,0)=( T_jh,0)
      $$


 The above computation is based on the assumption that $N_C$ is flat. We now assume that $N_C$ is {\it Hermitian and flat}, i.e. the $N_C$ admits locally constant Hermitian transition matrices.
Then all $\rho(e_{n+j})$ are Hermitian and constant matrices. Since they pairwise commute, they  are simultaneously diagonalizable.
Hence, we can assume that $M_j=\diag(\mu_{j,1},\ldots, \mu_{j,d})$.
   We  recall from \re{hTj} that $T_j=\diag(\la_{j,1},\dots, \la_{j,n})$ are already diagonal.

%
%
%
\begin{defn}\label{def-nr}
	The 
	normal bundle $N_C$ is said to be 
	non-resonant if, for each $(Q,P)\in \mathbb{N}^d\times \Z^n$ with $|Q|>1$,  each $i=1,\ldots, n$, and each $j=1,\ldots ,d$, there exist $i_h:=i_h(Q,P,i)$  and 
 $i_v:=i_v(Q,P,j)$ that are in $ \{1,\ldots, n\}$ such that
	$$
	\la_{i_h}^P 
\mu_{i_h}^Q-\la_{i_h,i 
}\neq 0\quad\text{and} 
\quad 
\la_{i_v}^P 
\mu_{i_v}^Q-\mu_{{i_v},j}\neq 0.
$$
\end{defn}

\begin{defn}\label{dioph}
The pullback 
normal bundle $N_C$, or $N_C$,  is said to be  {Diophantine} if for all $(Q,P)\in \mathbb{N}^d\times \Z^n$,   $|Q|>1$ and all $i=1,\ldots, n$, and  $j=1,\ldots ,d$,
\begin{gather}
\max_{\ell\in \{1,\ldots, n\}}\left |
\la_{\ell}^P 
\mu_{\ell}^Q
-\la_{\ell,i} 
\right|   >  \frac{D}{(|P|+|Q|)^{\tau}},\label{dh}\\
\max_{\ell\in \{1,\ldots, n\}}  \left |\la_\ell^P 
\mu_{\ell}^Q-\mu_{{\ell},j}\right |   >  \frac{D}{(|P|+|Q|)^{\tau}}.\label{dv}
\end{gather}
 We shall choose $i_h$ (resp. $i_v$) to be the index that realizes the maximum of \rea{dh} (resp. \rea{dv}).
\end{defn}
\begin{remark}
If  the right-hand sides are replaced by $0$, 
 then  
$N_C$ is non-resonant.
\end{remark}

\begin{prop}
	The properties of being non-resonant and Diophantine is a property of the (abelian) group and not the choice of the generators.
\end{prop}
\begin{proof}
Recall that $$
	\lambda_{\ell}=(\lambda_{\ell,1},\ldots,\lambda_{\ell,n})=(e^{2\pi i\tau_{\ell,1}},\ldots, e^{2\pi i\tau_{\ell,n}}).
	$$
Let $G$ be the group generated by the $\hat \tau_{\ell}$'s. Then, $\{\tilde \tau_{\ell}\}_{\ell}$ defines another set of generators of $G$ if $\tilde \tau_{\ell}=\hat\tau_{1}^{a_{\ell,1}}\cdots \hat\tau_{n}^{a_{\ell,n}}$, $\ell=1,\ldots, n$ where $A=(a_{i,j})_{1\leq i,j\leq n}\in GL_n(\Z)$ with $\det A=\pm1$. Then, the eigenvalues of $\tilde \tau_{\ell}$ are
$$
\tilde \lambda_{\ell,i}=\prod_{k=1}^n \lambda_{k,i}^{a_{\ell,k}},\quad \tilde \mu_{\ell,j}=\prod_{k=1}^{ d} \mu_{k,j}^{a_{\ell,k}}.
$$
Hence, we have
$$
\tilde \lambda_{\ell}^P\tilde\mu_{\ell}^Q\tilde \lambda_{\ell,i}^{-1}= (\lambda_{1}^P\mu_1^Q\lambda_{1,i}^{-1})^{a_{\ell,1}}\cdots (\lambda_{n}^P\mu_n^Q\lambda_{n,i}^{-1})^{a_{\ell,n}}.
$$
Fix $P,Q$ and $i$. Taking the logarithm, we have as $n$-vectors
\beq\label{log}
\left(\ln \tilde \lambda_{\ell}^P\tilde\mu_{\ell}^Q\tilde \lambda_{\ell,i}^{-1}\right)_{\ell=1,\dots, n}= A \left(\ln \lambda_{\ell}^P\mu_{\ell}^Q\lambda_{\ell,i}^{-1}\right)_{\ell=1,\dots, n},\mod 2\pi i.
\eeq
Since $A,A^{-1}
\in GL_n(\Z)$, given $P, Q$ and $i$, $$\left(\ln \tilde \lambda_{\ell}^P\tilde\mu_{\ell}^Q\tilde \lambda_{\ell,i}^{-1}\right)_{\ell=1,\dots, n}=0\mod 2\pi i$$
 iff $\left(\ln \lambda_{\ell}^P\mu_{\ell}^Q\lambda_{\ell,i}^{-1}\right)_{\ell=1,\dots, n}=0\mod 2\pi i$. Similarly, by considering $\ln \tilde\lambda_{\ell}^P\tilde\mu_{\ell}^Q\tilde \mu_{\ell,i}^{-1}$, we obtain that the non-resonant condition does not depend on the choice of generators.
Given $P,Q,i$,  if one of the $\lambda_{\ell}^P\mu_{\ell}^Q\lambda_{\ell,i}^{-1}$'s is not close to $1$, then $\left\|(\ln \tilde \lambda_{\ell}^P\tilde \mu_{\ell}^Q\tilde\lambda_{\ell,i}^{-1})_{\ell}\right\|$ is bounded way from $0$. On the other hand, if all $\lambda_{\ell}^P\mu_{\ell}^Q\lambda_{\ell,i}^{-1}$'s are close to $1$, then
$|\ln \lambda_{\ell}^P\mu_{\ell}^Q\lambda_{\ell,i}^{-1}|$ 
(with $\IM\ln$ in $(-\pi,\pi]$) is comparable to
$|\lambda_{\ell}^P\mu_{\ell}^Q\lambda_{\ell,i}^{-1}-1|$.
Furthermore, taking the module of \re{log}, we obtain
$$
\|A^{-1}\|^{-1}\left\|(\ln \lambda_{\ell}^P\mu_{\ell}^Q\lambda_{\ell,i}^{-1})_{\ell}\right\|\leq \left\|(\ln \tilde \lambda_{\ell}^P\tilde \mu_{\ell}^Q\tilde\lambda_{\ell,i}^{-1})_{\ell}\right\|\leq \|A\|\left\|(\ln \lambda_{\ell}^P\mu_{\ell}^Q\lambda_{\ell,i}^{-1})_{\ell}\right\|,
$$
where $\|(a_{\ell})_{\ell}\|=\max_{\ell}|a_{\ell}|$.
If the latter is bounded  below by $\frac{C}{(|P|+|Q|)^{\tau}}$, so is  $\left\|(\ln \tilde \lambda_{\ell}^P\tilde \mu_{\ell}^Q\tilde\lambda_{\ell,i}^{-1})_{\ell}\right\|$.
\end{proof}
\begin{thm}
	Let $C_n$ be an $n$-dimensional complex torus, holomorphically embedded into a complex manifold $M_{n+d}$.   Assume that $T_CM$ splits. Assume the normal bundle $N_C$ has (locally constant) Hermitian transition functions. Assume that  $N_C$ is   Diophantine. Then some neighborhood of $C$ is biholomorphic to a neighborhood of the zero section in the normal bundle.
\end{thm}
\begin{remark}
	When $C$ is a product of $1$-dimensional tori with normal bundle which is a direct sum of line bundles, the above result is due to Il'yashenko-Pyartli~\cite{ilyashenko-pyartly-embed}.
\end{remark}

We have
$\tau_j(h,v) =\hat\tau_j(h,v)+(\tau^h_j(h,v),\tau^v_j(h,v) )$. Here, functions
$$
\tau^{\bullet}_j(h,v)=\sum_{Q\in \mathbb{N}^d, |Q|\geq 2}\tau^{\bullet}_{j,Q}(h)v^Q
$$
are holomorphic in $(h,v)$ in a neighborhood of $\Om_{\epsilon,r}$ with values in $\C^{n+d}$.

\begin{defn}
	
Set
 $\Om_{\epsilon,r}:=\Om_\epsilon\times\Del_r^d$,
$
\tilde \Omega_{\epsilon,r}:=\overline\Omega_{\epsilon,r}\cup\bigcup_{i=1}^n \hat\tau_i(\overline\Omega_{\epsilon,r})$.
Denote by $\cL A_{\epsilon,r}$ $($resp. $\tilde{\cL A}_{\epsilon, r})$ the set of holomorphic functions on $\overline{\Omega_{\epsilon,r}}$ $($resp. $\overline{\tilde \Omega_{\epsilon,r}})$.
If $f\in \mathcal{A}_{\epsilon, r}$ $($resp. $\tilde f\in\tilde{\cL A}_{\epsilon, r})$, we set
$$
\|f\|_{\epsilon,r}:=\sup_{(h,v)\in\Omega_{\epsilon,r}}|  f(h,v)|, \quad |||\tilde f|||_{\epsilon,r}:=\sup_{(h,v)\in\tilde\Omega_{\epsilon,r}}| \tilde f(h,v)|.
$$
\end{defn}

As such, each
  $f\in \mathcal{A}_{\epsilon, r}$
can be expressed as a convergent Taylor-Laurent series
$$
f(h,v)= \sum_{P\in \Z^n}f_{Q,P}h^Pv^Q
$$
for $(h,v)\in \Om_{\epsilon,r}=\Om_\epsilon\times\Del_r^d$.
Recall that each holomorphic function on $E_j$ in \rl{connected} admits a unique Taylor-Laurent series expansion on $\tilde\Om_{\e'}^{1\cdots n}\times \Del_{\e'}^d$ when $\e'>0$ is sufficiently small.

We can state the following, for later use~:
\begin{lemma}\label{connected}
	For  $i\neq j$, the set $
\hat\tau_j\overline\Omega_{\epsilon,r}\cap\hat\tau_i
(\overline\Omega_{\epsilon,r})$ is a connected Reinhardt domain containing  $\tilde
\Om_{\e}^{1\dots n}\times \Del_{r'}^d$ in $\C^{n+d}$ when $r'>0$ is sufficiently small.
\end{lemma}
\subsection{Holomorphic functions on $\Omega_{\epsilon,r}$}
In this section, we study elementary properties and estimate  holomorphic functions   $f$ on $\overline\Omega_{\epsilon,r}$.
\begin{lemma}\label{laurent}
	An element  $f(h)$ of ${\mathcal A}_{\epsilon}$, that is a holomorphic function in a neighborhood of $\ov \Om_{\epsilon}$,  admits a Laurent series expansion in $h$
	\eq{L-exp}
	  f(h)=\sum_{P\in\Z^n} c_Ph^P.
	\eeq
	The series converges normally on $\Om_\epsilon$. Moreover, the Laurent coefficients
	\eq{L-coef}
	c_P=\f{1}{(2\pi i)^n}\int_{|\zeta_1|=s_1,\dots, |\zeta_n|=s_n}   f(\zeta)\zeta^{-P-(1,\dots, 1)}\, d\zeta_1\wedge\cdots \wedge d\zeta_n
	\eeq
	are independent of $s\in \Om_\epsilon^+$ and
	\eq{L-est}
	|c_P|\leq\sup_{\Om_\epsilon}| f|\inf_{s\in \Om_\epsilon^+} s^{-P}.
	\eeq
\end{lemma}
\begin{proof}Obviously, estimate \re{L-est} follows from \re{L-coef}.  Define $A(a,b)=\{w\in \C:a<|w|<b\}$. Fix $h\in \Om_\epsilon$. Then $|h|:=(|h_1|,\dots, |h_n|)\in \Om_\epsilon^+$.  The latter is an open set and we have for a small positive number $\e=\e(h)$,
	$$
	  f(h)=\f{1}{(2\pi i)^n}
	\int_{\pd A(|h_1|-\e,|h_1|+\e)}
	\cdots \int_{\pd A(|h_n|-\e,|h_n|+\e)}\f{  f(\zeta)\, d\zeta_1\dots d\zeta_n}{(\zeta_1-h_1)\cdots(\zeta_n-h_n)}.
	$$
 By  Laurent expansion in one-variable, we get the expansion \re{L-exp} in which $c_P$ are given by \re{L-coef} if we take $r_j=|h_j|\pm\e$ according to the sign of $p_j\in \Z^{\mp}$.
	
	We want to show that $c_P$ is independent of $s\in\Om_\epsilon^+$. Note that $\Om_\epsilon^+$ is a connected open set. For any two points $s,\tilde s$ can be connected by a union of line segments
in $\Om_\e^+$  which are parallel to coordinate axes in $\R^n$. 
 Using such line segments,  say a line segment $[a,b]\times (s_2,\dots, s_n)$ in $\Om_\e^+$, we know that $ f(\zeta)\zeta^{-P-(1,\dots, 1)}$ is holomorphic in $\zeta_1$ in the closure of $A(a,b)$ when $|\zeta_2|=s_2,\dots, |\zeta_n|=s_n$. By Cauchy theorem, the integrals are independent of $s_1\in[a,b]$ for these $s_2,\dots, s_n$. This shows that \re{L-coef} is independent of $s$.
	
	Finally the series converges uniformly on each compact subset of $\Om_\epsilon$. Indeed, for a small perturbation of $h$, we can choose $\e(h)$ to be independent of $h$. Then we can see easily that the series converges locally uniform in $h$.
\end{proof}
Recall that
\gan
{\mathcal P}_\epsilon^+=\left\{\sum_{i=1}^n t_i\IM\tau_i\in\R^n\colon t\in (-\epsilon,1+\epsilon)^{n}\right\},\\
{\Om}_\epsilon^+=\left\{(e^{-2\pi R_1},\dots, e^{-2\pi R_n})\colon R
\in {\mathcal P}_\epsilon^+
\right\}.
\end{gather*}
\begin{lemma}\label{dist-lemma} There is a constant $\kappa_0>0$ that depends only on $\IM\tau_1,\dots,\IM\tau_n$ such that if  $P\in\R^n$ and $\epsilon>\epsilon'$, there exists $R\in {\mathcal P}_{\epsilon}^+$ such that for all $R'  \in {\mathcal P}_{\epsilon'}^+$ we have
\eq{RR'}
  (R'-R)\cdot P\leq  -\kappa_0(\e-\e')|P|.
\eeq
\end{lemma}
\begin{proof}Without loss of generality, we may assume that $P$ is a unit vector.
Let $\pi(x)P$ be the orthogonal projection from $x\in\R^n$ onto the line spanned by $P$. Choose $R\in\ov {\mathcal P}_\epsilon^+$ so that $\pi(R)$ has the largest value for $R\in  \ov {\mathcal{P}_\epsilon^+}$. Note that $R$ must be on the boundary of ${\mathcal P}_\epsilon^+$ and latter is contained in the half-space $H$ defined by $\pi(y)\leq\pi(R)$ for $y\in\R^n$. Hence, $\pd H$ is orthogonal to $P$.    Then for any $R'\in \mathcal {P}_{\epsilon'}^+$,
$$
\pi(R)-\pi(R')=\dist(R',\pd H)\geq\dist ( \pd {\mathcal P}_{\epsilon}^+,{\mathcal P}_{\epsilon'}^+)\geq (\e-\e')/C.
$$
Therefore,   we obtain $(R'-R)\cdot P=-(\pi(R)-\pi(R'))\leq -(\e-\e')/C.$
\end{proof}
\begin{remark}Since $ {\mathcal P}_{\epsilon}^+$ is a   parallelotope, we can choose $R$ to be a vertex of ${\mathcal P}_{\epsilon}^+$.
\end{remark}
 In what follows, we denote the fixed constant~:
\beq\label{kappa}
\kappa:=2\pi\kappa_0.
\eeq
\begin{lemma}[Cauchy estimates]\label{cauchy}
	If $f\in\cL A_{\epsilon, r}$, then for all $(P,Q)\in \Z^n\times \mathbb {N}^d$,
	\beq\label{Reps}
	|f_{Q,P}|\leq \frac{M}{r^{|Q|}\max_{s\in \Om_\epsilon^+} s^{P}}.
	\eeq
	Furthermore, 
 if $f=\sum_{Q\in \nn^d, P\in \Z^n}f_{QP}h^Pv^Q\in \cL A_{\epsilon , r}$   
 and  $0<\del<{\kappa}\e$,
	then
	$$
	\|f\|_{\epsilon
  {-{\delta}/{\kappa}}, re^{-\delta}}\leq \frac{C\sup_{\Omega_{\epsilon,r}}|f|}{\del^{\nu}},
	$$
	where $C$ and $\nu$ depends only on $n$ and $d$.
\end{lemma}
\begin{proof}
	According to \rl{laurent} 
	and Cauchy estimates on polydiscs, we have for any $s\in \Om_\epsilon^+$,
$$
	|f_{Q,P}|\leq \frac{\sup_{\Om_\epsilon}|f_Q(h)|}{s^{P}}\leq \frac{\sup_{\Omega_{\epsilon,r}}|f|}{s^{P}r^{|Q|}}.
$$
	According to \rl{laurent} and Cauchy estimates for polydiscs, we have if $(h,v)\in \Omega_{\epsilon e^{-\del'},re^{-\delta}}$, then for all $s\in \Om_{\epsilon}^+$,
\begin{align*}
|f_Q(h)|&\leq \frac{\sup_{v\in  \Delta^d_r 
}|f(h,v)|}{r^{|Q|}},\\
|f_{Q,P}h^P|&\leq \left|\f{1}{(2\pi i)^n}\int_{|\zeta_1|=s_1,\dots, |\zeta_n|=s_n} f_Q(\zeta)\frac{h^P}{\zeta^{P}}\, \frac{d\zeta_1\wedge\cdots \wedge d\zeta_n}{\zeta_1\cdots\zeta_n}\right|.
\end{align*}
Set $s_j=e^{-2\pi R_j}$, $|h_j|=e^{-2\pi R_j'}$, $R=(R_1,\dots, R_n)$ and $R'=(R_1',\dots, R_n')$.  By \rl{dist-lemma},
\beq\label{fourier-type}
\inf_{(|\zeta_1|,\dots,|\zeta_n|)=s\in\Om_{\e}^+}\sup_{h\in \Om_{\e  {-\del'}}}
\left|\frac{h^P}{\zeta^P}\right|=\inf_{R\in \mathcal P_{\e}^+}\sup_{R'\in\mathcal P_{\e  {-\del'}}^+} e^{-2\pi <R-R',P>}\leq e^{-\kappa\delta' |P|},
\eeq
where the positive constant $\kappa$, defined by \rl{dist-lemma} and \re{kappa}, is independent of $P,\zeta_1,\zeta_2$.  Thus
\beq\label{est-small}
|f_{Q,P}h^Pv^Q|\leq  \frac{\sup_{\Omega_{\epsilon,r}}|f| e^{-\kappa\del'|P|}r^{|Q|}e^{-\delta|Q|}}{r^{|Q|}}\leq \sup_{\Omega_{\epsilon,r}}|f| e^{-\delta'\kappa |P|}e^{-\delta|Q|}.
\eeq
	Hence, setting $\del':=\delta/\kappa$, we have
	$$
	\|f\|_{\epsilon
  {-{\del}/{\kappa}}, re^{-\delta}}\leq \sum_{Q\in \nn^d, P\in \Z^n}|f_{Q,P}h^Pv^Q|\leq \frac{C\sup_{\Omega_{\epsilon,r}}|f|}{\del^{\nu}},
	$$
	where $C$ and $\nu$ depends only on $n$ and $d$.
\end{proof}
\subsection{Conjugacy  of the deck transformations}
Let us show that there is a biholomorphism $\Phi=(h,v)+\phi(h,v)$ of some neighborhood $\Omega_{\tilde\epsilon,\tilde r}$, fixing $\cyl$ pointwise (i.e. $\Phi(h,0)=(h,0)$) such that
$$
\Phi\circ \hat \tau_i = \tau_i\circ \Phi,\quad i=1,\ldots, n.
$$
This reads $\hat\tau_i+\phi(\hat\tau_i)=\hat \tau_i(Id+\phi)+\tau^{\bullet}_i(Id+\phi)$, that is, for all $i=1,\ldots n$,
\begin{equation}
\cL L_i(\phi)=\tau^{\bullet}_i(Id+\phi)+\left(\hat \tau_i(Id+\phi)- \hat \tau_i-D\hat \tau_i.\phi\right)=\tau^{\bullet}_i(Id+\phi).
\end{equation}
Here we define $\cL L_i(\phi):=\phi(\hat\tau_i)-D\hat \tau_i.\phi$ and
$ (\cL L^h_i(\phi^h),\cL L^v_i(\phi^v)):=\cL L_i(\phi)$. We have
$$
(\cL L^h_i(\phi^h),\cL L^v_i(\phi^v))=\left(\phi^h(T_ih, M_iv)-T_i\phi^h(h,v),  \phi^v(T_ih, M_iv)-M_i\phi^v(h,v)\right).
$$
  Since $\hat\tau_j$ are linear, then $D\hat\tau_j=\hat\tau_j$.
Expand the latter in Taylor-Laurent expansions as
\begin{eqnarray*}
\cL L_i^h(\phi^h)&=&\sum_{Q\in{\mathbb N}^d_2}\left(\sum_{P\in \Z^n}\left(\la^P_i\mu_i^Q\times Id_n-T_i\right)\phi^h_{Q,P}h^P\right)v^Q, \quad \phi^h_{Q,P}\in \C^n,\\
\cL L_i^v(\phi^v)&=&\sum_{Q\in{\mathbb N}^d_2}\left(\sum_{P\in \Z^n}\left(\la^P_i\mu_i^Q\times Id_d-M_i\right)\phi^v_{Q,P}h^P\right)v^Q, \quad \phi^v_{Q,P}\in \C^d.\\
\end{eqnarray*}
Recall the notations
$
\lambda_{\ell}=(\lambda_{\ell,1},\ldots,\lambda_{\ell,n})$ and $ \mu_\ell=(\mu_{\ell,1},\ldots,\mu_{\ell,d}).
$
With $P=(p_1,\ldots, p_n)\in\Z^n$ and $Q=(q_1,\ldots, q_d)\in \mathbb{N}^d$, we have
$$
\lambda_{\ell}^P\mu_{\ell}^Q:=\prod_{i=1}^n\lambda_{\ell,i}^{p_i}\prod_{j=1}^n\mu_{\ell,j}^{q_j}.
$$

\begin{lemma}\label{computeFij}Let $\tau_j\in\cL A_{\e,r}^{n+d}$ with
$
\tau_j=\hat\tau_j-F_j$ and $ F_j(h,v)=O(|v|^{q+1})
$
  with $q\geq1$.
Suppose that $\tau_i\tau_j=\tau_j\tau_i$ in a neighborhood of $\Om_{0}^{1\cdots n}\times\{0\}$ in $\C^{n+d}$.   Then
$
\cL L_iF_j-\cL L_iF_i=O(|v|^{2q+1}).
$\end{lemma}
\begin{proof}
 Recall that $\hat\tau_i(h,v)$ are linear maps in $h,v$.
Also,  $\hat\tau_i\hat\tau_j$ sends $\Om_{\e}^j\times\Del_\e^d$ into $\C^{n+d}$. Since  $\tau_j\in\cL A_{\e,r}^{n+d}$ and $\tau_j=\hat\tau_j+O(|v|^2)$, the continuity implies that $\tau_i\tau_j$ is well-defined on the product domain $\Om^{j}_{\e'}\times\Del_{r'}^d$ when $\e', r'$ are sufficiently small. Fix $h\in\Om^{j}_{\e'}$. By Taylor expansions in $v$, we obtain
$$
\tau_i\tau_j(h,v)=\hat\tau_i\hat\tau_j(h,v)+
\hat\tau_jF_j (h,v)+F_i\circ\hat\tau_j(h,v)+O(|v|^{2q+1}).
$$
Since $\hat\tau_i\hat\tau_j=\hat\tau_j\hat\tau_i$ and $\tau_i\tau_j=\tau_j\tau_i$ in a neighborhood of $\Om_{0}^{1\cdots n}\times\{0\}$, we get $\cL L_iF_j=\cL L_jF_i=O(|v|^{2q+1})$ in a possibly smaller neighborhood of $\Om_{0}^{1\cdots n}\times\{0\}$.
\end{proof}

  We will apply the following result to   $F_j=\hat\tau_j-J^{2q}(\tau_j)$,   where $J^{2q}(\tau_j)$ denotes the $2q$-jet at $0$ in the variable $v$. These are holomorphic on $\Om_{\e,r}$.
  Recall that
  $$
  \Omega^{ij}_{\epsilon',r'}:= \Omega_{\epsilon',r'}\cap\hat\tau_i^{-1}(\Omega_{\epsilon',r'})\cap\hat\tau_j^{-1}(\Omega_{\epsilon',r'})
  $$
\begin{prop}\label{cohomo}
		Assume   $N_C$
		is   
		Diophantine.   Fix $\epsilon_0,r_0,\del_0$  in $(0,1)$. Let $0<\e'<\epsilon<\epsilon_0$, $0<r'<r<r_0$, $0<\del<\del_0$, and $\frac{\del}{\kappa}<\e$. Suppose that $F_i\in {\cL A}_{\epsilon, r}$, $i=1,\ldots, n$,
 satisfy
		\beq
\label{almost-com}  \cL L_i(F_j)-\cL L_j(F_i)=0\quad
\text{on  }  \Omega^{ij}_{\epsilon',r'}.
		\eeq
There exist functions  $G \in  \tilde{\cL A}_{\epsilon
 {-{\del}/{\kappa}},re^{-\delta}}$ 
 such that
		 \begin{align}\label{formal2q+1}
		 \cL L_i(G)&=F_i \ \text{on   } \Omega_{\epsilon  {-{\del}/{\kappa}},re^{-\del}}.
				 \end{align}
	 Furthermore, $G$ satisfies
		\begin{align}\label{estim-sol}
		\|G\|_{\epsilon  {-{\del}/{\kappa}},re^{-\delta}}&\leq  \max_i\|F_{i}\|_{\epsilon,r}\frac{C'}{\delta^{\tau+\nu}},\\
\label{estim-solcompo}
		\|G\circ\hat\tau_i\|_{\epsilon  {-{\del}/{\kappa}},re^{-\delta}}&\leq  \max_i\|F_{i}\|_{\epsilon,r}\frac{  C'}{\delta^{\tau+\nu}}.
		\end{align}
		for some constant $\kappa, C'$  that are independent of $F,q,\del, r,\e$ and $\nu$ that depends only on $n$ and $d$.
Furthermore, if $F_j(h,v)=O(|v|^{q+1})$ for all $j$, then
\eq{}
G(h,v)=O(|v|^{q+1}), \quad G(h,v)=J^{2q}G(h,v).
\eeq
\end{prop}

\begin{proof}
Since $F_i\in {\cL A}_{\epsilon, r}$, we can write
$$
F_i(h,v)=\sum_{Q\in \mathbb{N}^d, |Q|\geq2}\sum_{P\in\Z^n}F_{i,Q,P}h^Pv^Q,
$$
which converges normally   for $(h,v)\in\Om_{\epsilon,r}$. We emphasize that $F_{i,Q,P}$
are vectors, and its $k$th component
 is denoted  by $F_{i,k,Q,P}$.
For each $(Q,P)\in \mathbb{N}^d\times \Z^n$,   each $i=1,\ldots, n$, and each $j=1,\ldots ,d$, let $i_h:=i_h(Q,P,i), i_v:=i_v(Q,P,j)$ be in $ \{1,\ldots, n\}$ as in \rd{def-nr}.
Let us set
\begin{align}
	G^h_i&:=\sum_{Q\in \mathbb{N}^d, 2\leq|Q|\leq 2q}\sum_{P\in\Z}\frac{F^h_{i_h,i,Q,P}}{\lambda_{i_h}^P\mu_{i_h}^Q-\lambda_{i_h,i}}h^Pv^Q,\quad i=1,\ldots, n\label{sol-coh-h}\\
	G^v_j&:=\sum_{Q\in \mathbb{N}^d, 2\leq |Q|\leq 2q}\sum_{P\in\Z}\frac{F^v_{i_v,j,Q,P}}{\lambda_{i_v}^P\lambda_{i_v}^Q-\mu_{i_v,j}}h^Pv^Q,\quad j=1,\ldots. d\label{sol-coh-v}.
\end{align}
According to \re{almost-com}, we have
\beq\label{compat-coeff}
(\lambda_{i_h}^P\mu_{i_h}^Q-\lambda_{i_h,i})F_{m,i,Q,P}^h=(\lambda_{m}^P\mu_{m}^Q-\lambda_{m,i}) F_{i_h,i,Q,P}^h, \quad 2\leq|Q|\leq 2q.
\eeq
Therefore, using \re{compat-coeff}, the $i$th-component of $\cL L_m(G)$ reads
\begin{align*}
	\cL L_m(G^h_i) = &\sum_{Q\in \mathbb{N}^d, 2\leq |Q|\leq 2q }\sum_{P\in \Z^n}(\lambda_{m}^P\mu_{m}^Q-\lambda_{m,i})\frac{F^h_{i_h,i,Q,P}}{(\lambda_{i_h}^P\mu_{i_h}^Q-\lambda_{i_h,i})}h^Pv^Q\\
	= &\sum_{Q\in \mathbb{N}^d,2\leq|Q|\leq 2q}\sum_{P\in \Z^n}F_{m,i,Q,P}^hh^Pv^Q.
\end{align*}
Proceeding similarly for the vertical component, we have obtained, the formal equality~:
\beq\label{formal-lin}
 \cL L_m(G) = F_m 
,\quad m=1,\ldots, n.
\eeq
Let us estimate these solutions.
According to \rd{dioph} and formulas \re{sol-coh-h}-\re{sol-coh-v}, we have
\eq{Ghv}
\max_{i,j}\left (|G^h_{i,Q,P}|,|G^v_{j,Q,P}|\right )\leq \max_i|F_{i,Q,P}|\frac{(|P|+|Q|)^{\tau}}{D}.
\eeq
Let $(h,v)\in \Om_{\epsilon  {-{\del}/{\kappa}},re^{-\delta}}$. According to \re{est-small}, 
we have
\aln
\|G^h_{Q,P}h^Pv^Q\|&\leq \max_i\|F_{i}\|_{\epsilon,r}e^{-\delta(|P|+|Q|)}\frac{(|P|+|Q|)^{\tau}}{D}\\
&\leq \max_i\|F_{i}\|_{\epsilon,r}e^{-\delta/2(|P|+|Q|)}\frac{(2\tau e)^{\tau}}{D\delta^{\tau}}.
\end{align*}
Summing over $P$ and $Q$, we obtain
$$
\|G\|_{\epsilon  {-{\del}/{\kappa}},re^{-\delta}}\leq \max_i\|F_{i}\|_{\epsilon,r}\frac{C'}{\delta^{\tau+\nu}},
$$
for some constants $C',\nu$ that are independent of $F,\epsilon,\del$.
 Hence, $G\in
 \cL A_{\epsilon  {-{\del}/{\kappa}},re^{-\delta}}$. 

Let us prove \re{estim-solcompo}.   Let $B:=2\max_{k,i,j}(|\lambda_{k,i}|, |\mu_{k,j}|)$. Then, there is a constant $D'$ such that
\ga\label{sd-ehanced}
\max_{\ell\in \{1,\ldots, n\}}\left |\lambda_{\ell}^P\mu_{\ell}^Q-\lambda_{\ell,i}\right|
\geq  \frac{D'\max_k|\lambda_{k}^P\mu_{k}^Q|}{(|P|+[Q|)^{\tau}},\\
\label{sd-ehanced+}
\max_{\ell\in \{1,\ldots, n\}}  \left |\lambda_{\ell}^P\mu_{\ell}^Q-\mu_{{\ell},j}\right | \geq  \frac{D'\max_k|\lambda_{k}^P\mu_{k}^Q|}{(|P|+[Q|)^{\tau}}.
\end{gather}
Indeed, if $\max_k|\lambda_k^P\mu_k^Q|<B$, then \rd{dioph} gives \re{sd-ehanced} with $D':= \frac{D}{B}$. Otherwise, if $$|\lambda_{k_0}^P\mu_{k_0}^Q|:=\max_k|\lambda_{k}^P\mu_{k}^Q|\geq B,$$
then $|\lambda_{k_0,i}|\leq \frac{B}{2}\leq \frac{|\lambda_{k_0}^P\mu_{k_0}^Q|}{2}$. Hence, we have
$$
\left |\lambda_{k_0}^P\mu_{k_0}^Q-\lambda_{k_0,i}\right|\geq \left||\lambda_{k_0}^P\mu_{k_0}^Q|- |\lambda_{k_0,i}|\right|\geq \frac{|\lambda_{k_0}^P\mu_{k_0}^Q|}{2}.
$$
We have verified \re{sd-ehanced}. Similarly, we can verified \re{sd-ehanced+}.
Finally, combining all cases gives us, for $m=1,\ldots, n$,
\begin{align*}
|[G\circ\hat \tau_m]_{QP}| &= \left| G_{Q,P}\lambda_{m}^P\mu_{m}^Q
\right|\leq\max_{\ell}|F_{\ell,Q,P}| \frac{|\lambda_{m}^P\mu_{m}^Q|}{|\lambda_{i_h}^P\mu_{i_h}^Q-\lambda_{i_h,i}|}\\
&\leq \max_{\ell}|F_{\ell,Q,P}|\frac{|\lambda_{m}^P\mu_{m}^Q|(|P|+|Q|)^{\tau}}{D'\max_k|\lambda_{k}^P\mu_{k}^Q|}\\
&\leq  \max_{\ell}|F_{\ell,Q,P}|\frac{(|P|+|Q|)^{\tau}}{D'}.\end{align*}
 Hence,  $\tilde G_m:=G\circ \hat \tau_m\in \cL A_{\epsilon  {-{\del}/{\kappa}},re^{-\delta}}$.   We can define  $\tilde G\in \tilde{\cL A}_{\epsilon  {-{\del}/{\kappa}},re^{-\delta}}$ such that $\tilde G=\tilde G_m\hat\tau_m^{-1}$ on $\hat\tau_m\Om_{\e,r}$.  We verify that $\tilde G$ extends to a single-valued holomorphic function of class $\tilde {\cL A}_{\e,r}$. Indeed,
  $\tilde G_i\hat\tau_i^{-1}=\tilde G_j\hat\tau_j^{-1}$ on $\hat\tau_i\Om_{\e,r}\cap \hat\tau_j\Om_{\e,r}$, since the latter is connected by \rl{connected} and the two functions agree with $G$ on $\hat\tau_i\Om_{\e,r}\cap \hat\tau_j\Om_{\e,r}\cap \Om_{\e,r}$ that contains a neighborhood of $\tilde\Om_\e\times\{0\}$ in $\C^{n+d}$.
%
\end{proof}

  In what follows, we shall set $\hat\tau_0=Id$ and for any 
   $f \in (\tilde A_{\epsilon,r})^{
   n+d}$ and $F\in (\tilde{\cL  A}_{\epsilon,r})^{n+d}$, we set
\begin{align*}
|||f|||_{\epsilon,r}&:=\|f\|_{\epsilon,r}+\sum_{i=1}^n\| \hat\tau_i^{-1}
f\circ \hat\tau_i\|_{\epsilon,r}= \sum_{i=0}^n\|  \hat\tau_i^{-1}
f\circ \hat\tau_i\|_{\epsilon,r},\\
F_{(i)}&:=\hat\tau_i^{-1}F\circ\hat\tau_i\in \cL A_{\epsilon,r},\quad F_{(0)}:=F,\quad i=1,\ldots, n.
\end{align*}

\subsection{Iteration scheme}
We shall prove the main result through a Newton scheme.
Let us define sequences of positive real numbers
$$
\del_{k}:=
  \frac{\del_0}{(k+1)^2}, 
\quad r_{k+1}:=r_ke^{-5\del_k}, \quad \epsilon_{k+1}:= \epsilon_k
{ {-\frac{5\del_k}{\kappa}}}
$$
such that
\beq\label{sigma}
\sigma:=\sum_{k\geq 0}\del_k <2\del_0 
\eeq
We assume the following conditions hold~:
\begin{gather}
	 \label{cond1}\del_0<\frac{\kappa}{20}\e_0,\\
	\label{cond2}\del_0<\frac{\ln 2}{10}.
\end{gather}
Condition \re{cond1} ensures that
$$
\frac{5}{\kappa}\sigma<\frac{10}{\kappa}\del_0<\frac{\e_0}{2},\quad \text{so that}\quad \e_{k}>\frac{\e_0}{2},\; k\geq 0.
$$
Condition \re{cond2} ensures that
$$
e^{-5\sigma}>e^{-10\del_0}>\frac{1}{2},\quad \text{so that}\quad r_k>\frac{r_0}{2},\;  k\geq 0.
$$
Let  $m=5$ be  fixed. We define $\epsilon_{k+1}<\epsilon_{k}^{(\ell)}<\epsilon_{k}$ and $r_{k+1}<r_{k}^{(\ell)}<r_{k}$ 
, $\ell=1,\ldots m$  as follows~:
\begin{align*}
\epsilon_{k}^{(\ell)}=\epsilon_{k} {-\frac{\ell\del_k}{\kappa}},&\quad\epsilon_{k+1}:=\epsilon_{k}^{(m)},\\
r_k^{(\ell)}=r_ke^{-\ell\del_k},&\quad r_{k+1}:=r_k^{(m)} .
\end{align*}
 We emphasize that condition \re{cond1} ensures $\epsilon_{k}^{(\ell)}>0$.
 Let us assume that for each $i=1,\ldots, n$, $\tau_i^{(k)}=\hat \tau_i+ \tau^{\bullet (k)}_i,$ is holomorphic and on $\Om_{\epsilon_{k},r_k}$ it satisfies
 \eq{4.32-15f}
  ||\tau^{\bullet(k)}||_{\epsilon_{k},r_k}<\del_k^{\mu}, \quad \tau^{\bullet(k)}(h,v)=O(|v|^{q_k+1}).
  \eeq
 We further assume that $\tau_i^{(k)},\tau_j^{(k)}$ commute on
 $\Om^{ij}_{\epsilon',r'}
 \subset \Om_{\epsilon_{k},r_k}$
 for some positive
 $\e'<\e_k,r'<r_k$.
   We take
 $$
  q_{k+1}=2q_k+1\geq q_02^{k}
 $$
 for $q_0\geq
 1$ to be determined.

We will define a sequence $\Phi^{(k)}$ with   $\Phi^{(k)}(h,0)=(h,0)$. Let us write on appropriate domains
\begin{gather*}
\Phi^{(k)}=Id+\phi^{(k)},\quad (\Phi^{(k)})^{-1}:=Id-\psi^{(k)},\\
\quad \tau_i^{(k+1)} = \Phi^{(k)}\circ \tau_i^{(k)}\circ (\Phi^{(k)})^{-1},\quad i=1,\ldots, n.
\end{gather*}

Note that there is a constant $C>0$ (depending only on the $\mu_{i,j}$'s) such that if   $\e''<\e'$, $Cr''<r'$ then
$$
\Om^{ij}_{\e'',r''}\subset\Om_{\e'}\times\Del_{r'}^d,\quad \Om_{\e''}\times\Del_{r''}^d\subset \Om^{ij}_{\e',r'}.
$$
Since the $\tau_i^{(k)}, \tau_j^{(k)}$   commute   on $\Om_{\e'}\times\Del_{\e'}^d$ for some $\e'>0$ and $\Phi^{(k)}(h,v)=(h,v)+O(|v|^2)$, then   $\tau_i^{(k+1)}, \tau_j^{(k+1)}$ still commute on the same kind of domains for a possibly smaller $\e'$.
  By \rl{computeFij}, we obtain :
\begin{align}\label{compatible}
\cL L_j(\tau^{\bullet (k)}_i)-\cL L_i(\tau^{\bullet (k)}_j)=&
O(|v|^{2q_k+1}).
\end{align}

We want to find $\phi^{(k)}$ so that
\begin{align}
	\label{tik+1}
	\tau_i^{(k+1)}:=\hat\tau_i+\tau^{\bullet (k+1)}_i&=\Phi^{(k)}\circ\tau_i^{(k)}\circ (\Phi^{(k)})^{-1}\\
	&=\hat \tau_i(Id-\psi^{(k)})+\tau^{\bullet(k)}_i(Id-\psi^{(k)})\nonumber\\
	&\quad +\phi^{(k)}\left(\hat \tau_i(Id-\psi^{(k)})+\tau^{\bullet(k)}_i(Id-\psi^{(k)})\right)
	\nonumber\end{align}
is defined and bounded on $ \Om_{\epsilon_{k+1},r_{k+1}}$.

 We now define $\Phi^{(k)}$ by applying \rp{cohomo} with these  $F_i:=-J^{2q_k}(\tau^{\bullet (k)}_i)$. Let $\phi^{(k)}$ stands for $G$.   
 Therefore, given $0<\del<\kappa \epsilon_k^{(1)}$, $\phi^{(k)}$ is holomorphic and bounded on $\tilde\Om_{\epsilon_{k} {-\frac{\del}{\kappa}},r_ke^{-\del}}$
 and it satisfies on $\Om_{\epsilon_{k}^{(1)} {-\frac{\del}{\kappa}},r_k^{(1)}e^{-\del}}$,
\beq\label{sol-cohom-approx}
  \cL L_i(\phi^{(k)}):=\phi^{(k)}(\hat\tau_i)-D\hat \tau_i.\phi^{(k)}= -J^{2q_k}(\tau^{\bullet (k)}_i),\quad i=1,\ldots,n.
\eeq
Writing formally $\Phi^{(k)}\circ \Psi^{(k)}=Id$ and using linearity of $\hat\tau_i$, we obtain
\begin{align}
\psi^{(k)}_{(i)} &= \phi^{(k)}_{(i)}(Id-\psi^{(k)}_{(i)}), \quad i=0,\ldots, n,\label{contr-i}
\end{align}
recalling the notation $\phi^{(k)}_{(i)}=\hat\tau_i^{-1}\phi^{(k)}\circ \hat\tau_i$.
According to \rp{cohomo}, we have
\beq\label{phik}
|||\phi^{(k)}|||_{\epsilon_k  {-\frac{\del}{\kappa}},r_ke^{-\delta}}\leq ||\tau^{\bullet(k)}||_{\epsilon_k,r_k}\frac{C'}{\delta^{\tau+\nu}}\leq C'\delta_k^{\mu}\del^{-\tau-\nu}.
\eeq
 We recall that the constatn $C'$ does not depend on $k$, nor on $\tau^{\bullet(k)}$.
\begin{lemma}\label{l-inclusion}   There is constant $\tilde D>0$ (independent of $k$) such that, given a positive number $\mu$,
so that  $\del< \min\{\tilde D,\ln 2,\frac{r_0}{2}\}$  satisfies
 \eq{2C'}
 2^{m+2}C'\delta_k^{\mu}\del^{-\tau-\nu-3}<1,
  \eeq
where  $m=5$
. Then, for all $0\leq \ell \leq m$ and $i=0,\ldots,n$, the maps $\Phi^{(k)}_{(i)}:=Id+\phi^{(k)}_{(i)}$ are biholomorphisms
	$$
	\Phi^{(k)}_{(i)}\colon \Omega_{\epsilon_k {-\frac{(\ell+1)\del}{\kappa}},r_ke^{-(\ell+1)\delta}}\to \Omega_{\epsilon_k  {-\frac{\ell\del}{2\kappa}},r_ke^{-\frac{\ell\delta}{2}}}
	$$
with inverse $\Psi^{(k)}_{(i)}:=Id-\psi^{(k)}_{(i)} ~:\Omega_{\epsilon_k {-\frac{(\ell+2)\del}{\kappa}},r_ke^{- (\ell+2)\delta}}\to \Omega_{\epsilon_k {-\frac{(\ell+1)\del}{\kappa}},r_ke^{-(\ell+1)\delta}}$ satisfying
	$$\Phi^{(k)}_{(i)}\circ \Psi^{(k)}_{(i)}=\hat\tau_i^{-1}(\Phi^{(k)}\circ\Psi^{(k)})\circ\hat\tau_i=Id$$ on $\Omega_{\epsilon_k {-\frac{(\ell+2)\del}{\kappa}},r_ke^{-(\ell+2)\delta}}$,
provided $q_0>C(\del_0,\mu)$. \end{lemma}
\begin{proof}

We have $1-e^{-\del}> \del/2$   and $\frac{1}{2}<e^{-\del}$ for $0<\del<\ln 2$. Assuming  $\delta<\frac{r_0}{2}$, we have $\del<r_k$
so that
$$\dist(\Delta^d 
_{r_{k}e^{-( 1 
+\ell)\del}},\partial \Delta^d 
_{r_ke^{-\ell\del}})=r_ke^{-\ell\del}(1-e^{-  
\delta})>\frac{ 
\delta^2}{2^{\ell+1}}.$$
On the other hand,
by \re{boundary-of-R}, we have
$$\dist(\Om_{\epsilon_{k} {-\frac{(\ell+  1 
)\del}{\kappa}}},\partial \Om_{\epsilon_{k} {-\frac{\ell\del}{\kappa}}})=\dist(\Om_{\epsilon_{k}
 {-\frac{(\ell+  1 
 )\del}{\kappa}}}^+,\partial \Om_{\epsilon_{k} {-\frac{\ell\del}{\kappa}}}^+).$$
Let $(e^{-2\pi R},e^{-2\pi R'})\in \Om_{\epsilon_{k} {-\frac{(\ell+1)\del}{\kappa}}}^+\times \partial\Om_{\epsilon_{k} {-\frac{\ell\del}{\kappa}}}^+$.  
Since the matrix $(\IM\tau_i^j)_{1\leq i,j\leq n}$ is invertible, there is a constant $\tilde C$ (independent of $k$) such that
$$
|e^{-2\pi R}-e^{-2\pi R'}|\geq \tilde C^{-1}\left| 1-e^{-2\pi (R- R')}\right|>2\pi\tilde C^{-1}|R-R'|\geq  2\pi\tilde C^{-1}\frac{\delta}{\kappa}.
$$
 Let us set $\tilde D:= 2\frac{2\pi}{\tilde C\kappa}$.
Assuming $\del<\tilde D$, we have $\del<\tilde D<2^{\ell+2}\tilde D\leq 2^{m+2}\tilde D$. Assume $2^{m+2}C'\delta_k^{\mu}\del^{-\tau-\nu-3}<1$. Then according to \re{phik}, we have, for  $(h,v)\in \Omega_{\epsilon_k {-\frac{\ell\del}{\kappa}},r_ke^{-\ell\delta}}$,
\al{}\label{dist-min}
|\phi^{(k)}_{(i)}(h,v)|&\leq C'\delta_k^{\mu}\del^{-\tau-\nu}< \frac{\delta^3}{2^{m+2}}<\frac{\del}{2}\frac{\delta^2}{2^{\ell+1}} \\
&< \frac{\del}{2}\dist(\Omega_{\epsilon_k {-\frac{(\ell+ 1 
)\del}{\kappa}},r_ke^{-(\ell+ 1 
)\delta}},\partial \Omega_{\epsilon_k {-\frac{\ell\del}{\kappa}},r_ke^{-\ell\delta}}).\nonumber
\end{align}
By the Cauchy inequality, we have, for $(h,v)\in \Omega_{\epsilon_k {- \frac{(\ell+ 2 
)\del}{\kappa}},r_ke^{-(\ell+ 2 
)\delta}}$
\eq{Dphi}
|D\phi^{(k)}_{(i)}(h,v)|\leq \frac{C'\delta_k^{\mu}\del^{-\tau-\nu}}{\dist(\Omega_{\epsilon_k {-\frac{(\ell+ 2 
)\del}{\kappa}},r_ke^{-(\ell+  2 
)\delta}},\partial \Omega_{\epsilon_k {-\frac{ (\ell+1)\del}{\kappa}},r_ke^{- (\ell+1)\delta}})}\leq  \frac{\del}{2}
<1/2.
\eeq
We can apply the contraction mapping theorem to \re{contr-i}  together with the last inequality of \re{dist-min}. We find a holomorphic solution $\psi^{(k)}_{(i)}$ such that  for $(h,v)\in \Omega_{\epsilon_k {- \frac{(\ell+ 3 
)\del}{\kappa}},r_ke^{-(\ell+  3
)\delta}}$
\begin{align}
|\psi^{(k)}_{(i)}(h,v)|&\leq  ||\phi^{(k)}_{(i)}||_{\epsilon_k  {-\frac{\del\crd
(\ell+
2)}{\kappa}},r_ke^{- (\ell+2)
\delta}}
\leq  C' 
 \delta_k^{\mu}\del^{-\tau-\nu}
 \leq  \frac{\del^3}{2^{m+2}}\label{estimpsi}\\
&<\frac{\del}{2^{m-\ell-1}}\dist(\Omega_{\epsilon_k {-\frac{(\ell+3)\del}{\kappa}},r_ke^{-(\ell+3)\delta}},\partial \Omega_{\epsilon_k {-\frac{(\ell+2)\del}{\kappa}},r_ke^{-(\ell+2)\delta}}).\nonumber
\end{align}
Hence, we have found a mapping $\Psi^{(k)}_{(i)}:=Id-\psi^{(k)}_{(i)}$ such that
$$
\hat\Omega_{\epsilon_k {-\frac{(\ell+3)\del}{\kappa}},r_ke^{-(\ell+3)\delta}}:=\Psi^{(k)}_{(i)}(\Omega_{\epsilon_k {-\frac{(\ell+3)\del}{\kappa}},r_ke^{-(\ell+3)\delta}})\subset \Omega_{\epsilon_k {-\frac{(\ell+2)\del}{\kappa}},r_ke^{-(\ell+2)\delta}}.
$$
Also, $\Phi^{(k)}_{(i)}\circ\Psi^{(k)}_{(i)}=Id$ 
  on   $\Omega_{\epsilon_k {-\frac{(\ell+3)\del}{\kappa}},r_ke^{-(\ell+3)\delta}}$.
Therefore,
$$
\Phi^{(k)}_{(i)}\colon\hat \Omega_{\epsilon_k {-\frac{(\ell+3)\del}{\kappa}},r_ke^{- (\ell+3)\delta}}\to  \Omega_{\epsilon_k {-\frac{(\ell
+1)\del}{\kappa}},r_ke^{-(\ell
+1)\delta}}
$$
is an (onto) biholomorphism such that $(\Phi^{(k)}_{(i)})^{-1}=\Psi^{(k)}_{(i)}$.
\end{proof}
\begin{prop}\label{prop4.20}
Keep conditions on $\del,\mu$ in \rla{l-inclusion} as well as conditions \rea{cond1},\rea{cond2}.
	If $\del_0$ small enough   there is possibly larger
 $\mu>0$  such that if for all $i=1,\ldots, n$, $\tau_i^{(0)}\in  ({\cL A}_{\epsilon_{0},r_{0}})^{n+d}$ with $|||\tau^{\bullet(0)}_i|||_{\epsilon_{0},r_{0}}\leq \del_{0}^{\mu}$, then for  all $k\geq 0$  we have the following~:
\begin{align}
\tau_i^{\bullet(k+1)}\in ({\cL A}_{\epsilon_{k+1},r_{k+1}})^{n+d}, \quad &\tau_i^{\bullet(k+1)}=O(|v|^{2q_k+1}),\nonumber\\
||\tau^{\bullet(k+1)}_i||_{\epsilon_{k+1},r_{k+1}}&\leq \del_{k+1}^{\mu}.\label{tauDotk}
\end{align}
%
%
\end{prop}
\begin{proof}
Let us first show that $\tau^{(k+1)}_{i}:=\Phi^{(k)}\circ \tau^{(k)}_{i}\circ(\Phi^{(k)})^{-1}$
is well defined on $\Omega_{\epsilon_k {- \frac{\ell\del}{\kappa}},r_ke^{-\ell\delta}}$ for $m\geq\ell\geq 5$ and all $i=1,\ldots, n$.  Here $m$ is fixed from \rl{l-inclusion}.
Indeed, we have
$$
\tau^{(k+1)}_{i}= \hat\tau_i(I+\phi^{(k)}_{(i)})\circ (I+\hat\tau_i^{-1}\tau_{i}^{\bullet(k)})\circ (\Phi^{(k)})^{-1}.
$$
 Since $\tau_i^{\bullet(k)}$ is of order $\geq 2q_{k-1}+1$, Schwarz inequality gives
$$
\|\hat\tau_i^{-1}\tau^{\bullet(k)}_i\|_{\epsilon_k  {- \frac{(\ell-1)\del}{\kappa}},r_ke^{-(\ell-1)\delta}}\leq \max_i \|\hat\tau_i^{-1}\|_{\e_0,r_0}C e^{-(2q_{k-1}+1)\del}\|\tau^{\bullet(k)}_i\|_{\epsilon_k  {- \frac{(\ell-2)\del}{\kappa}},r_ke^{-(\ell-2)\delta}}.
$$
We recall that $\del_0$ satisfies \re{cond1} and \re{cond2}. Setting $\del:=\del_k$ and if $q_0$ large enough, we have
$$
\max_i \|\hat\tau_i^{-1}\|_{\e_0,r_0}C e^{-(2q_{k-1}+1)\del}\leq \max_i \|\hat\tau_i^{-1}\|_{\e_0,r_0}C e^{-\frac{2^{k-1}}{(k+1)^2}q_0\del_0}<1.
$$
According to \rl{l-inclusion},  \re{2C'} and the distance estimate in \re{dist-min}, we have
$$
(I+\hat\tau_i^{-1}\tau_i^{\bullet(k)})\circ (\Phi^{(k)})^{-1}(\Omega_{\epsilon_k {- \frac{\ell\del}{\kappa}},r_ke^{-\ell\delta}})\subset \Omega_{\epsilon_k {- \frac{(\ell-3)\del}{\kappa}},r_ke^{-(\ell-3)\delta}}
$$
so that $\tau^{(k+1)}_i$ is defined on $\Omega_{\epsilon_k  {- \frac{\ell\del}{\kappa}},r_ke^{-\ell\delta}}$ for  $\ell
=4\leq m$ since  $\phi^{(k)}_{(i)}$ is
defined on $\Om_{\e_k-\del/\kappa,r_ke^{-\del/\kappa}}$.
For the rest of the proof, we fix  $\ell=4$.
 From the argument above, we have
$$
\tau_i^{(k+1)}(\Om_{\e_k-4\del/\kappa,r_ke^{-4\del/\kappa}})\subset \hat\tau_i(\Phi_{(i)}^{k}(\Om_{\e_k-\del/\kappa,r_ke^{-\del/\kappa}}))\subset \hat\tau_i(\Om_{\e_k,r_k})\subset \hat\tau_i(\Om_{\e_0,r_0}).
$$
Hence, $\tau_i^{\bullet(k+1)}$ is uniformly bounded on $\Om_{\e_k-4\del/\kappa,r_ke^{-4\del/\kappa}}$ w.r.t $k$~:
$$
\|\tau^{\bullet (k+1)}_i\|_{\e_k {-4\del/\kappa,r_ke^{-4\del}}}
\leq C.
$$

On the other hand, on $\Om_{\e_k-4\del/\kappa,r_ke^{-4\del/\kappa}}$, we have
\begin{align}
	\tau^{\bullet (k+1)}_i&=\hat \tau_i(\phi^{(k)}-\psi^{(k)})+\left(\tau^{\bullet(k)}_i(Id-\psi^{(k)})-\tau^{\bullet(k)}_i\right)\label{remainder1}\\
	&\quad+\left(\phi^{(k)}\left(\hat \tau_i-\hat \tau_i\psi^{(k)}+\tau^{\bullet(k)}_i(Id-\psi^{(k)})\right)-\phi^{(k)}(\hat \tau_i)\right)
\nonumber 
\\
	&\quad+(\phi^{(k)}(\hat \tau_i)-\hat \tau_i\phi^{(k)}+\tau^{\bullet(k)}_i).\nonumber
\end{align}
  The last term is equal to $\tau^{\bullet(k)}_i-J^{2q_k}(\tau^{\bullet(k)}_i)=O(|v|^{2q_k+1})$. 
We also have $\phi^{(k)}-\psi^{(k)}=O(|v|^{2q_k+1})$, $\phi^{(k)}=O(|v|^{q_k+1})$. Thus
\eq{tauk+1=2q+1}
\tau^{\bullet (k+1)}_i=O(|v|^{2q_k+1}).
\eeq

Improving the estimate by the Schwarz inequality, we obtain
$$
\|\tau^{\bullet(k+1)}_i\|_{\epsilon_k  {- \frac{5\del}{\kappa}},r_ke^{-5\delta}}\leq C e^{-q_{k+1}\del}.
$$
We have $e^{-x}<1/x$ for $x>0$. For   $\del:=\del_k< \min\{\tilde D,\ln 2,\frac{r_0}{2}\}$,   let $\mu$ satisfy $2^{m+2}C'\delta_k^{\mu-\tau-\nu-3}<1$, in which case assumptions of \rl{l-inclusion} are satisfied. We then obtain
$$
C\del_{k+1}^{-\mu}e^{-q_{k+1}\del} \leq C \left(\f{(k+2)^2}{\del_0}\right) ^{\mu +1}\f{1}{ q_02^{k}}<1
$$
provided $q_0>C(\del_0,\mu)$.
\end{proof}
Finally, quite classically, since \re{Dphi} and \re{sigma}, the sequence of diffeomorphisms $\{\Phi_k\circ\Phi_{k-1}\circ\cdots\circ \Phi_1\}_k$ converges uniformly on the open set $\Omega_{\frac{\epsilon_{0}}{2},r_{0}e^{-\sigma}}$  to a diffeomorphism $\Phi$ which satisfies
$$
\Phi\circ \hat \tau_i = \tau_i\circ \Phi,\quad i=1,\ldots, n
$$
  provided $q_0\geq C(\del_0,\mu)$ and    $ ||\tau^{\bullet(0)}_i||_{\epsilon_{0},r_{0}}\leq \del_{0}^{\mu}$ for some $\del_0,\mu$ are fixed. The condition $q_0>C(\del_0,\mu,\tau,\nu)$ can be achieved by using finitely many $\Phi_0,\dots, \Phi_{m}$.  Then initial condition $ ||\tau^{\bullet(0)}_i||_{\epsilon_{0},r_{0}}\leq \del_{0}^{\mu}$   can be achieved easily by a dilation in the $v$ variable. Indeed, we apply the dilation in $v$ to the original $\tau_1,\dots,\tau_n$ with $q_0=1$. This allows us to construct $\Phi_0$ in \rp{prop4.20} with $k=0$ and define $\Phi_0\tau_j\Phi_0^{-1}$ to achieve  $q_1\geq 2$. Then $\Phi_0\tau_j\Phi_0^{-1}$, $j=1,\dots, n$ still commute pairwise on $\Om_{\e'}\times\Del_{\e'}^d$ for some $\e'>0$. Applying the procedure again, this allows us to find $\Phi_1,\dots, \Phi_{k-1}$ to achieve $q_k\geq 2^{k}>  C(\del_0,\mu)$. Finally using dilation we can apply the full version of \rp{prop4.20} for all $k$ to construction a new sequence of desired mapping $\Phi_0,\Phi_1,\dots.$
Hence, the torus $C$ has a neighborhood in $M$ biholomorphic to a neighborhood of the zero section in its normal bundle $N_C$  since there is a biholomorphism fixing $\cyl$ that conjugate the deck transformations of the covering of the latter to those of the former.

 \begin{remark}
	The assumption "$T_CM$ splits" can be replaced by the condition that for all $P\in \Z^n$, for all $Q\in \mathbb{N}^d$, $|Q|=1$ and for all $i=1,\ldots, n$
	$$
	\max_{\ell\in \{1,\ldots, n\}}\left |\la_\ell^P 
\mu_{\ell}^Q-\la_{\ell,i} 
\right|   >  \frac{D}{(|P|+1)^{\tau}}.
	$$
\end{remark}

\bibliographystyle{alpha}
\bibliography{geometry,stolo}
\end{document}